\documentclass{amsart}
\usepackage[utf8]{inputenc}
\usepackage{amssymb}
\usepackage{hyperref}
\usepackage[final]{showkeys} %[final] is the option that removes them

\input xy
\xyoption{all}

\theoremstyle{definition}
\newtheorem{mydef}{Definition}[section]
\newtheorem{lem}[mydef]{Lemma}
\newtheorem{thm}[mydef]{Theorem}

\newtheorem{cor}[mydef]{Corollary}

\newtheorem{question}[mydef]{Question}

\newtheorem{defin}[mydef]{Definition}
\newtheorem{example}[mydef]{Example}
\newtheorem{remark}[mydef]{Remark}

\newtheorem{fact}[mydef]{Fact}

\newcommand{\fct}[2]{{}^{#1}#2}

\newcommand{\ba}{\bar{a}}
\newcommand{\bb}{\bar{b}}

\newcommand{\by}{\bar{y}}

\newcommand{\sea}{\mathfrak{C}}

\newcommand{\cf}[1]{\text{cf} (#1)}
\newcommand{\ccl}{\mathbf{cl}}
\newcommand{\cclp}[1]{\ccl_{#1}}
\newcommand{\seq}[1]{\langle #1 \rangle}

\newcommand{\id}{\text{id}}

\newcommand{\leap}[1]{\le_{#1}}
\newcommand{\ltap}[1]{<_{#1}}

\newcommand{\lta}{\ltap{\K}}
\newcommand{\lea}{\leap{\K}}

\newcommand\cq{\mathcal {Q}}
\newcommand\Set{\operatorname{\bf Set}}
\newcommand\colim{\operatorname{colim}}
\newcommand\cp{\mathcal {P}}

\newcommand{\K}{\mathbf{K}}

\newbox\noforkbox \newdimen\forklinewidth
\forklinewidth=0.3pt \setbox0\hbox{$\textstyle\smile$}
\setbox1\hbox to \wd0{\hfil\vrule width \forklinewidth depth-2pt
 height 10pt \hfil}
\wd1=0 cm \setbox\noforkbox\hbox{\lower 2pt\box1\lower
2pt\box0\relax}
\def\unionstick{\mathop{\copy\noforkbox}\limits}

\def\1nf{\unionstick^{(1)}}

\def\2nf{\unionstick^{(2)}}
\def\3nf{\unionstick^{(3)}}

\newcommand{\Ll}{\mathbb{L}}
\newcommand{\otp}{\operatorname{otp}}

\newcommand{\ck}{{\mathcal K}}
\newcommand{\Hilb}{{\bf Hilb}}

\newcommand{\tlt}{\triangleleft}
\newcommand{\tleq}{\trianglelefteq}

\newcommand{\cl}{\ccl}

\newcommand{\LS}{\text{LS}}
\newcommand{\OR}{\text{OR}}

\newcommand{\BI}{\mathbf{I}}
\newcommand{\BJ}{\mathbf{J}}
\newcommand{\SCH}{\text{SCH}}

\newcommand{\Def}{\operatorname{Def}}

\title{Internal sizes in $\mu$-abstract elementary classes}
\date{\today
}

\author[Lieberman]{Michael Lieberman}
\email{lieberman@math.muni.cz}
\urladdr{http://www.math.muni.cz/\textasciitilde lieberman/}
\address{Department of Mathematics and Statistics, Faculty of Science, Masaryk University, Brno, Czech Republic}

\author[Rosick\'y]{Ji\v r\'i Rosick\'y}
\email{rosicky@math.muni.cz}
\urladdr{http://www.math.muni.cz/\textasciitilde rosicky/}
\address{Department of Mathematics and Statistics, Faculty of Science, Masaryk University, Brno, Czech Republic}
\thanks{The first and second authors are supported by the Grant Agency of the Czech Republic under the grant P201/12/G028.}

\author[Vasey]{Sebastien Vasey}
\email{sebv@math.harvard.edu}
\urladdr{http://math.harvard.edu/\textasciitilde sebv/}
\address{Department of Mathematics \\ Harvard University \\ Cambridge, Massachusetts, USA}

\begin{document}

\begin{abstract}
 Working in the context of $\mu$-abstract elementary classes ($\mu$-AECs)---or, equivalently, accessible categories with all morphisms monomorphisms---we examine the two natural notions of size that occur, namely cardinality of underlying sets and \emph{internal size}.  The latter, purely category-theoretic, notion generalizes e.g.\ density character in complete metric spaces and cardinality of orthogonal bases in Hilbert spaces.  We consider the relationship between these notions under mild set-theoretic hypotheses, including weakenings of the singular cardinal hypothesis.  We also establish preliminary results on the existence and categoricity spectra of $\mu$-AECs, including specific examples showing dramatic failures of the eventual categoricity conjecture (with categoricity defined using cardinality) in $\mu$-AECs.
 
  %The internal size of an object $M$ inside a given category is, roughly, the least infinite cardinal $\lambda$ such that any morphism from $M$ into the colimit of a $\lambda^+$-directed system factors through one of the components of the system. The existence spectrum of a category is the class of cardinals $\lambda$ such that the category has an object of internal size $\lambda$. We study the existence spectrum in $\mu$-abstract elementary classes ($\mu$-AECs), which are, up to equivalence of categories, the same as accessible categories with all morphisms monomorphisms. We show for example that, assuming instances of the singular cardinal hypothesis which follow from a large cardinal axiom, $\mu$-AECs which admit intersections have objects of all sufficiently large internal sizes. We also investigate the relationship between internal sizes and cardinalities and analyze a series of examples, including one of Shelah---a certain class of sufficiently-closed constructible models of set theory---which show that the categoricity spectrum can behave very differently depending on whether we look at categoricity in cardinalities or in internal sizes.
\end{abstract}

\maketitle

%\tableofcontents

\section{Introduction}

There are a number of known connections between accessible categories and abstract model theory. Per \cite{lrcaec-jsl}, for example, abstract elementary classes (AECs) and their metric analogues (mAECs) can be characterized as accessible categories with directed colimits, with a faithful functor to the category of sets satisfying certain natural properties.  More recently, \cite{mu-aec-jpaa} develops the notion of a $\mu$-AEC---a generalization of AECs in which closure under unions of chains is replaced with closure under $\mu$-directed colimits---and shows that they are, up to equivalence of categories, precisely the accessible categories all of whose morphisms are monomorphisms.  The model-theoretic motivations for this generalization are manifold: for example, the class of $\mu$-saturated models of a given elementary class need not be an AEC, but is a $\mu$-AEC.  The equivalence between $\mu$-AECs and accessible categories with monomorphisms, moreover, allows the easy application of model-theoretic methods to accessible categories, and vice versa, and serves as the foundation in \cite{multipres-pams} for an emerging set of correspondences between $\mu$-AECs with natural additional properties (universal $\mu$-AECs, $\mu$-AECs admitting intersections) and accessible categories with added structure (locally multipresentable categories, locally polypresentable categories).

The aim of the present paper is to analyze the \emph{existence spectrum} of a general $\mu$-AEC $\ck$ and, to a lesser extent, those satisfying the additional closure conditions of the form described above.  By the existence spectrum, we mean the class of cardinals $\lambda$ such that $\ck$ contains \emph{at least one} object of size $\lambda$. We will also mention the \emph{categoricity spectrum} of $\ck$: the class of cardinals $\lambda$ such that $\ck$ contains \emph{exactly one} object of size $\lambda$, up to isomorphism.

At this point it is essential to note that a $\mu$-AEC $\ck$ comes equipped with \emph{two} natural notions of size: first, and most obviously, the cardinality $|U M|$ of the underlying set $U M$ of an object $M$, but also the {\it internal size} $|M|_{\ck}$ of $M$, which is derived from the purely category-theoretic {\it presentability rank}, $r_{\ck} (M)$, of $M$. Recall that the presentability rank of $M$ is the least regular cardinal $\lambda$ such that for any morphism $f: M \to N$ and any $\lambda$-directed system $\seq{N_i : i \in I}$ whose colimit is $N$, the map $f$ factors essentially uniquely through some $N_i$. Assuming that $|U M| > \LS (\ck)$, as we will throughout the text (even in AECs, the behavior of internal sizes below $\LS (\ck)$ seems hard to control, see Examples \ref{metric-example} or \ref{wo-example}), if $\ck$ has directed colimits (\cite[4.2]{beke-rosicky}) or if we assume GCH (\cite[2.3(5)]{beke-rosicky}) then the presentability rank of $M$ is always a successor cardinal, say $r_\ck(M)=\lambda^+$.  In this case, we define $|M|_\ck=\lambda$.  This notion of internal size matches nicely with the intuitive ones in familiar categories: if $\ck$ is an AEC, then $|M|_\ck=|UM|$, \cite[4.3]{lieberman-categ}; if $\ck$ is the category of Hilbert spaces and isometries, $|M|_\ck$ is the size of an orthonormal basis of $M$, Example \ref{hilbert-example} below; if $\ck$ is a mAEC, $|M|_\ck=\mbox{dc}(UM)$, the density character of the underlying complete metric space of $M$, \cite[3.1]{lrcaec-jsl}; and so on.  The case of mAECs---which are $\aleph_1$-AECs---is instructive: cardinalities and internal sizes will usually disagree when the latter is of countable cofinality.

The crux of this article consists of the analysis of the delicate relationship between these two notions of size, and the differences in the existence and categoricity spectra depending on the particular notion of size that we adopt.  Broadly speaking, we observe that the spectra are much smoother in the case of internal sizes, while gaps and even meaningful failures of eventual categoricity are easily constructed when we work instead with simple cardinality of models.

Let us make a bolder statement. In \cite{sh1019}, Shelah discusses whether there is a real model theory for infinitary quantification, i.e.\ going beyond AECs. We argue that the answer is positive but we should replace cardinalities of underlying sets by category-theoretic internal sizes. We note that this shift has already occurred---naturally, with no input from category theory---in continuous model theory. For example, the classes of complete metric structures considered in \cite{shus837} satisfy a Morley-like eventual categoricity result phrased in terms of the notion of internal size particular to that context, i.e.\ density character (\cite[8.2]{shus837}).

Going back to the present paper, we consider the following questions concerning the existence spectrum:

\begin{enumerate}
\item Under what conditions is the presentability rank a successor?
\item When does a $\mu$-AEC (or, more generally, an accessible category) have an object in all sufficiently large internal sizes?
\item Given an object $M$ of a $\mu$-AEC $\ck$, when can one give a simple description of the internal size of $M$? For example, when does $|U M|$ coincide with $|M|_{\ck}$?
\end{enumerate}

Regarding the first question, we have already mentioned that the presentability rank is a successor in any accessible category provided that GCH holds. In this paper, we weaken the assumption of GCH to SCH (the singular cardinal hypothesis). Since SCH is known to hold above a strongly compact cardinal \cite{solovay-sch} (and indeed above an $\omega_1$-strongly compact \cite{bagaria-magidor}), we obtain from a large cardinal axiom that any sufficiently large presentability rank is a successor. Interestingly, the strongly compact cardinals have been used several times before in the classification theory for AECs \cite{makkaishelah, tamelc-jsl, bg-apal, stab-spec-jml} and the study of $\mu$-AECs \cite{mu-aec-jpaa}. More generally, we give a category-theoretic condition implying that limit presentability rank cannot exist (Theorem \ref{successor-thm}). We show that this condition holds (in ZFC) in any $\mu$-AEC admitting intersections (Corollary \ref{inter-succ}). We do not know of any examples of an accessible category where sufficiently large presentability ranks are not successors.

The second question above was originally asked by Beke and the second author \cite{beke-rosicky}. As there, let us call an accessible category $\ck$ \emph{LS-accessible} if there is a cardinal $\lambda$ such that $\ck$ has objects of every internal size $\lambda' \ge \lambda$. It is open whether a large accessible category is always LS-accessible. Several conditions suffice to guarantee that this holds. For example it is true if the category has products (\cite[4.7(2)]{beke-rosicky}) or coproducts (\cite[4.7(3)]{beke-rosicky}), and true in any $\mu$-AEC provided it has directed colimits (\cite[2.7]{ct-accessible-jsl}).  We show here that, under the assumption of a suitable instance of SCH, any $\mu$-AEC admitting intersections and with arbitrarily large models is LS-accessible (Theorem~\ref{intersec-ls-acc}).  Put another way, under this set-theoretic assumption any large accessible category with wide pullbacks (i.e.\ a large locally polypresentable category) and all morphisms monomorphisms is LS-accessible (see \cite[5.7]{multipres-pams}). We show, moreover, that even without assuming any instance of SCH, such $\mu$-AECs are weakly LS-accessible: they contain an object of internal size $\lambda$ for any sufficiently large \emph{regular} cardinal $\lambda$. In fact, we can show this purely category-theoretically for the locally multipresentable categories (without assuming that all morphisms are monomorphisms), see Theorem \ref{lsacc} in the Appendix.

The third question, namely the relationship between internal size and cardinality, is studied in Section \ref{mu-aecs-sec}. We show there that in a $\mu$-AEC $\ck$ we have $|M|_\ck=|UM|$ (so in particular the presentability rank is a successor) whenever $\lambda=|UM|$ is $\mu$-closed---$\theta^{<\mu}<\lambda$ for all $\theta<\lambda$---but that the sizes may disagree otherwise.  Under GCH (or, indeed, appropriate instances of SCH), this inequality simplifies drastically, yielding Theorem~\ref{internal-size-mu-aec}: If $\lambda=|UM|$ is not $\lambda_0^+$ with $\cf{\lambda_0} < \mu$, then $|M|_\ck=|UM|$, else $|M|_\ck$ is either $\lambda_0$ or $\lambda$. In the case of $\mu$-AECs admitting intersections, one can give an even simpler description of the internal size: it is the minimum cardinality of a subset $A$ of $U M$ such that $M$ is minimal over $A$ (Theorem \ref{internal-cl}).

In Section \ref{examples-sec} we give several examples, primarily concerning the following two generalizations of Shelah's eventual categoricity conjecture to $\mu$-AECS given in Section 6 of \cite{mu-aec-jpaa}:

\begin{question}[Eventual categoricity, cardinality]\label{categ-q-1}
	If a $\mu$-AEC is categorical in \emph{some} sufficiently large cardinal $\lambda$ with $\lambda=\lambda^{<\mu}$, is it categorical in \emph{all} sufficiently large $\lambda'$ such that $\lambda'=(\lambda')^{<\mu}$?
\end{question}
\begin{question}[Eventual categoricity, internal size]\label{categ-q-2}
	If a $\mu$-AEC is categorical in internal size $\lambda$ for some sufficiently large $\lambda$, is it categorical in every sufficiently large internal size?
\end{question}

Shelah's eventual categoricity conjecture for ($\aleph_0$-)AECs is one of the main open questions of classification theory for non-elementary classes and it is natural to ask about broader frameworks than AECs where it may fail. This was the motivation for the two questions above.

In Section \ref{shelah-example}, we carefully analyze an example of Shelah (see the introduction of \cite{sh1019}) in the hope of obtaining a counterexample to Question \ref{categ-q-2}.  The particular $\aleph_1$-AEC involved consists of the class of well-founded models of Kripke-Platek set theory isomorphic to $(L_\alpha, \in)$ such that for all $\beta < \alpha$, $[L_\beta]^{\le \aleph_0} \cap L \subseteq L_\alpha$. This example fails to contradict Question \ref{categ-q-2}, but in an instructive way: while---assuming $V=L$---it is categorical in power in any cardinal of countable cofinality and has many models in all other powers (see Theorem \ref{card-nm}), it is nowhere categorical in the sense of internal sizes (Corollary \ref{internal-nm}). This is because, roughly speaking, some of the abundance of models of cardinality $\lambda^+$ with $\cf{\lambda}=\aleph_0$ drop down to internal size $\lambda$---a concrete illustration of the pathology captured in Theorem~\ref{internal-size-mu-aec}, and a suggestion of the smoothing effect that comes with the passage to internal sizes.  This example doubles, incidentally, as an example of the drastic extent to which the spectrum for categoricity in power of a $\mu$-AEC may depend on the ambient set theory, e.g. on the existence or nonexistence of $0^{\sharp}$.

Note that this example does not contradict Question \ref{categ-q-1} either, as categoricity occurs only in cardinals of countable cofinality.  We can, however, give a very simple and mathematically natural counterexample to Question \ref{categ-q-1}, namely the $\aleph_1$-AEC $\ck$ of Hilbert spaces and isometries (Example \ref{hilbert-example}).  Here the internal sizes correspond, as mentioned above, to the cardinality of orthonormal bases, meaning that $\ck$ is everywhere categorical in the sense of internal sizes.  In terms of cardinalities, though, $\ck$ has (assuming GCH for simplicity) exactly two models of cardinality $\lambda^+$ for each $\lambda$ with $\cf{\lambda}=\aleph_0$---one of internal size $\lambda$ and one of internal size $\lambda^+$---with exactly one model of cardinality $\lambda$ for all other $\lambda$ of uncountable cofinality.  There is, perhaps, a broader lesson to be drawn: the case of Hilbert spaces strongly suggests that \emph{the formulation of eventual categoricity in terms of internal size is, in a sense, the more mathematically honest}: it is a mere quirk of fate that, depending on our ambient set theory, there may be literally as many models as we like of a particular cardinal of uncountable cofinality (GCH limits this number to two), but the essential character of each such model is determined by its internal size, i.e. the size of its basis, which can be pinned down in ZFC.

We assume a knowledge of the basic concepts related to accessible categories and $\mu$-AECs (we still briefly give the relevant background definitions at the beginning of each section). Comprehensive accounts of the former can be found in \cite{adamek-rosicky} and \cite{makkai-pare}, while the current state of knowledge concerning the latter is summarized in \cite{mu-aec-jpaa} and \cite{multipres-pams}.

We thank Will Boney, Rami Grossberg, and the referee, for comments that helped improve this paper. 

\section{Set-theoretic preliminaries}\label{set-theory-sec}

For an infinite cardinal $\lambda$, let $\lambda_r$ denote the least regular cardinal greater than or equal to $\lambda$. That is, $\lambda_r$ is $\lambda^+$ if $\lambda$ is singular or $\lambda$ otherwise. We also define:

\begin{defin}\label{star-def}
$$
  \lambda^\ast := \begin{cases}
    \lambda^+ & \text{if } \lambda \text{ is a successor cardinal} \\
    \lambda & \text{if } \lambda \text{ is a limit cardinal}
  \end{cases}
$$
\end{defin}

The notion of a $\mu$-closed cardinal will be used often:

\begin{defin}\label{sch-def}
  Let $\mu \le \lambda$ be infinite cardinals.
  \begin{enumerate}
  \item We say that $\lambda$ is \emph{$\mu$-closed} if $\theta^{<\mu} < \lambda$ for all $\theta < \lambda$.
  \item We say that $\lambda$ is \emph{almost $\mu$-closed} if $\theta^{<\mu} \le \lambda$ for all $\theta < \lambda$.
  \item For $S$ a class of infinite cardinals greater than or equal to $\mu$, we write $\SCH_{\mu, S}$ for the statement ``every $\lambda \in S$ is almost $\mu$-closed''. $\SCH_{\mu, \ge \lambda}$ has the obvious meaning.
  \item We write $\SCH_{\mu, \lambda}$ for the statement ``There exists a set $S \subseteq \lambda$ of cardinals that is unbounded in $\lambda$ and such that $\SCH_{\mu, S}$ holds''.
  \end{enumerate}
\end{defin}

\begin{remark}\label{sch-rmk}
  Assuming the generalized continuum hypothesis (GCH), an infinite cardinal $\lambda$ is $\mu$-closed (for $\mu$ regular) if and only if $\lambda \ge \mu$ and $\lambda$ is not the successor of a cardinal of cofinality less than $\mu$. Therefore in this case all cardinals greater than or equal to $\mu = 2^{<\mu}$ are almost $\mu$-closed. In fact the singular cardinal hypothesis\footnote{That is, for every infinite singular cardinal $\lambda$, $\lambda^{\cf{\lambda}} = 2^{\cf{\lambda}} + \lambda^+$.} is equivalent to $\SCH_{\mu, \ge 2^{<\mu}}$ for all regular cardinals $\mu$: see \cite[5.22]{jechbook} and Lemma \ref{sch-lem} below. Further if $\kappa$ is a strongly compact cardinal, then by a result of Solovay \cite[20.8]{jechbook}, $\SCH_{\mu, \ge (2^{<\mu} + \kappa)}$ holds for all regular cardinals $\mu$. In fact, this holds even if $\kappa$ is only $\omega_1$-strongly compact \cite[4.2]{bagaria-magidor}. However, it is consistent (assuming a large cardinal axiom) that there is no $\theta$ so that $\SCH_{\aleph_1, \ge \theta}$, see \cite{gch-foreman-woodin}. In fact, for a fixed sufficiently high $\lambda$, even the failure of $\SCH_{\aleph_1, \lambda}$ is consistent \cite[1.13]{gitik-magidor}.
\end{remark}

The following ordering between cardinals is introduced in \cite[2.3.1]{makkai-pare}. The reason for its appearance is that for $\mu \le \lambda$ regular, $\mu \tleq \lambda$ if and only if any $\mu$-accessible category is also $\lambda$-accessible.

\begin{defin}\label{triangle-def}
  Let $\mu \le \lambda$ be regular cardinals. We write $\mu \tlt \lambda$ if $\mu < \lambda$ and for any $\theta < \lambda$, $\cf{[\theta]^{<\mu}, \subseteq} < \lambda$. Here, $[\theta]^{<\mu}$ denotes the set of all subsets of $\theta$ of size less than $\mu$. Write $\mu \tleq \lambda$ if $\mu = \lambda$ or $\mu \tlt \lambda$.
\end{defin}

We will use the following characterization of $\tleq$. One direction appears in \cite[2.3.4]{makkai-pare} and the other in \cite[4.11]{lrcaec-jsl}. We sketch a short proof here for the convenience of the reader.

\begin{fact}\label{triangle-fact}
  Let $\mu < \lambda$ be regular cardinals. If $\lambda$ is $\mu$-closed, then $\mu \tlt \lambda$. Conversely, if $\lambda > 2^{<\mu}$ and $\mu \tlt \lambda$ then $\lambda$ is $\mu$-closed.
\end{fact}
\begin{proof}

It is well-known that for $\mu \le \theta$ with $\mu$ regular, we have the following identity, from which the result follows easily:

$$
\theta^{<\mu} = \cf{[\theta]^{<\mu}, \subseteq} \cdot 2^{<\mu}
$$

To see the identity, first observe that $\cf{[\theta]^{<\mu}, \subseteq} \cdot 2^{<\mu} \le \theta^{<\mu} \cdot 2^{<\mu} = \theta^{<\mu}$. Conversely, fix a cofinal family $\mathcal{F} \subseteq [\theta]^{<\mu}$ of cardinality $\cf{[\theta]^{<\mu}, \subseteq}$ and build an injection $f : [\theta]^{<\mu} \to \mathcal{F} \times \fct{<\mu}{2}$ (where $\fct{<\mu}{2}$ is the set of functions from an ordinal $\alpha < \mu$ to $2 = \{0, 1\}$) as follows: given $S \in [\theta]^{<\mu}$, pick $F = F_S \in \mathcal{F}$ such that $S \subseteq F$ and enumerate $F$ in increasing order as $\seq{\alpha_i : i < \otp (F)}$. Then let $\eta_{S, F} : \otp (F) \to 2$ be defined by $\eta_{S, F} (i) = 1$ if $\alpha_i \in S$ and $\eta_{S, F} (i) = 0$ otherwise. Set $f (S) = (F_S, \eta_{S, F})$. It is clear that from $F_S$ and $\eta_{S, F}$ one can recover $S$, hence $f$ is indeed and injection witnessing that $\theta^{<\mu} \le \cf{[\theta]^{<\mu}, \subseteq} \cdot 2^{<\mu}$.
\end{proof}

We end this section with some easy lemmas on computing cardinal exponentiation assuming instances of SCH.

\begin{lem}\label{sch-lem}
  Let $\mu \le \lambda$ be cardinals with $\mu$ regular.

  \begin{enumerate}
  \item If $\cf{\lambda} \ge \mu$, then $\lambda^{<\mu}$ is the least almost $\mu$-closed $\lambda' \ge \lambda$.
  \item If $\cf{\lambda} < \mu$, then $\lambda^{<\mu}$ is the least almost $\mu$-closed $\lambda' > \lambda$.
  \end{enumerate}

  In particular, if $\SCH_{\mu, \{\lambda, \lambda^+\}}$, then 
  \begin{align*}\lambda^{<\mu} = 
  \begin{cases}
  	\lambda & \text{ if }\,\cf{\lambda} \ge \mu\\
  	\lambda^+ & \text{ if }\,\cf{\lambda} < \mu
  \end{cases}	
  \end{align*}
\end{lem}
\begin{proof} \
  \begin{enumerate}
  \item Given that $\cf{\lambda} \ge \mu$, it is easy to check that $\lambda^{<\mu} = \sup_{\theta < \lambda} \theta^{<\mu}$. If $\lambda$ is almost $\mu$-closed then it follows that $\lambda = \lambda^{<\mu}$. If $\lambda$ is \emph{not} almost $\mu$-closed, there exists $\theta < \lambda$ such that $\lambda < \theta^{<\mu}$. Since $\left(\theta^{<\mu}\right)^{<\mu} = \theta^{<\mu}$, it is easy to check that $\theta^{<\mu}$ must be the least almost $\mu$-closed cardinal above $\lambda$ and moreover $\lambda^{<\mu} = \theta^{<\mu}$. 
  \item Check that $\lambda^{<\mu}$ is almost $\mu$-closed.
  \end{enumerate}
\end{proof}

\begin{lem}\label{mu-closed-sch}
  Let $\mu \le \lambda$ be cardinals with $\mu$ regular. The following are equivalent:

  \begin{enumerate}
  \item $\lambda$ is $\mu$-closed.
  \item $\SCH_{\mu, \lambda}$ and $\lambda$ is not the successor of a cardinal of cofinality less than $\mu$.
  \end{enumerate}
\end{lem}
\begin{proof}
  If $\lambda$ is $\mu$-closed, then (in ZFC), $\lambda$ cannot be the successor of a cardinal of cofinality less than $\mu$. Moreover, given $\theta < \lambda$, then $\theta^{<\mu}  < \lambda$ and $\theta^{<\mu}$ is almost $\mu$-closed. Therefore $\SCH_{\mu, \lambda}$. 

  Conversely, assume that $\lambda$ is not the successor of a cardinal of cofinality less than $\mu$ and $\SCH_{\mu, \lambda}$ holds. Let $\theta < \lambda$ be arbitrary. If $\cf{\theta} \ge \mu$, then by Lemma \ref{sch-lem} and the definition of $\SCH_{\mu, \lambda}$,  $\theta^{<\mu} < \lambda$. If $\cf{\theta} < \mu$, then since $\lambda \neq \theta^+$ by assumption we have that $\theta^+ < \lambda$. Thus again by Lemma \ref{sch-lem} and $\SCH_{\mu, \lambda}$, $\theta^{<\mu} < \lambda$. Thus $\lambda$ is $\mu$-closed, as desired.
\end{proof}

\section{Presentability in accessible categories}\label{foundations-sec}

For a cardinal $\mu$, a partially ordered set is called \emph{$\mu$-directed} if each of its subsets of cardinality strictly less than $\mu$ has an upper bound. Note that any non-empty poset is $2$-directed, $\aleph_0$-directed is equivalent to $3$-directed, and $\mu$-directed is equivalent to $\mu_r$-directed. Thus we will usually assume that $\mu$ is a regular cardinal.

For $\lambda$ a regular cardinal, we call an object $M$ of a category $\ck$ \emph{$\lambda$-presentable} if its hom-functor $\ck(M,-):\ck\to\Set$ preserves $\lambda$-directed colimits, i.e. colimits indexed by $\lambda$-directed sets. Put another way, $M$ is $\lambda$-presentable if for any morphism $f:M\to N$ with $N$ a $\lambda$-directed colimit $\langle \phi_\alpha:N_\alpha\to N\rangle$ with diagram maps $\phi_{\beta\alpha}:N_\alpha\to N_\beta$, $f$ factors essentially uniquely through one of the $N_\alpha$.  That is, $f=\phi_\alpha f_\alpha$ for some $f_\alpha:M\to N_\alpha$, and if $f=\phi_\beta f_\beta$ as well, there is $\gamma>\alpha,\beta$ such that $\phi_{\gamma\alpha}f_\alpha=\phi_{\gamma\beta}f_\beta$.

For $\lambda$ an infinite cardinal, we call an object $M$ of a category $\ck$ \emph{$(<\lambda)$-presentable} if it is $\lambda_0$-presentable for some regular $\lambda_0 < \lambda + \aleph_1$. Note that, for $\lambda$ regular, $(<\lambda^+)$-presentable is the same as $\lambda$-presentable. We will use the following key parameterized notions:

\begin{defin}
  Let $\ck$ be a category and let $\mu, \lambda$ be infinite cardinals with $\mu$ regular. A \emph{$(\mu, <\lambda)$-system} is a $\mu$-directed system consisting of $(<\lambda)$-presentable objects. Similarly, define \emph{$(\mu, \lambda)$-system} (for $\lambda$ regular).
\end{defin}

\begin{defin}
  Let $\ck$ be a category and let $\mu, \lambda$ be infinite cardinals with $\mu$ regular. We say an object $M$ of $\ck$ is \emph{$(\mu, <\lambda)$-resolvable} if it is the colimit of a $(\mu, <\lambda)$-system. We say that $\ck$ is \emph{$(\mu, <\lambda)$-resolvable} if all its object are $(\mu, <\lambda)$-resolvable. Similarly, define $(\mu, \lambda)$-resolvable (for $\lambda$ regular).
\end{defin}

To investigate presentability, the following notion is useful:

\begin{defin}\label{proper-def}
  A $(\mu, <\lambda)$-system with colimit $M$ is \emph{proper} if the identity map on $M$ does \emph{not} factor essentially uniquely through an element of the system. That is, if the system has diagram maps $\phi_{\beta \alpha} : M_\alpha \rightarrow M_\beta$ and colimit maps $\phi_{\alpha}: M_\alpha \rightarrow M$, then the following does \emph{not} occur: there exists $\alpha$ and $f_\alpha : M \rightarrow M_\alpha$ such that $\id_{M} = \phi_\alpha f_\alpha$ and whenever $\id_M = \phi_\beta f_\beta$ for $f_\beta : M \rightarrow M_\beta$, then there is $\gamma > \alpha, \beta$ such that $\phi_{\gamma \alpha} f_\alpha = \phi_{\gamma \beta} f_\beta$.
\end{defin}

We can see when a system is proper by looking at the presentability of its colimit. This is is well known and goes back to the fact that split subobjects of $\lambda$-presentables are $\lambda$-presentable (see for example \cite[1.3]{adamek-rosicky}). The proof is included for the convenience of the reader unacquainted with presentability. 

\begin{lem}\label{system-resolvable}
  Let $\mu, \lambda$ be infinite cardinals with $\mu$ regular and let $\ck$ be a category. Let $M \in \ck$ be the colimit of some $(\mu, <\lambda)$-system $\BI$. Then:

  \begin{enumerate}
  \item If $\BI$ is proper, then $M$ is not $\mu$-presentable.
  \item If $M$ is not $(<\lambda)$-presentable, then $\BI$ is proper.
  \end{enumerate}
\end{lem}
\begin{proof} \
  \begin{enumerate}
  \item If $M$ is $\mu$-presentable, then since $\BI$ is $\mu$-directed, $\id_M$ must factor essentially uniquely through an element of $\BI$, so $\BI$ is not proper by definition.
  \item Say $\BI$ has diagram maps $\seq{\phi_{\beta \alpha} : M_\alpha \rightarrow M_\beta, \alpha, \beta \in I}$ and colimit maps $\phi_\alpha : M_\alpha \rightarrow M$. Assume that $\BI$ is not proper and let $f_\alpha : M \rightarrow M_\alpha$ be such that $\phi_\alpha f_\alpha = \id_M$. By hypothesis, $M_\alpha$ is $(<\lambda)$-presentable, hence $\lambda_0$-presentable for some regular $\lambda_0 < \lambda + \aleph_1$. We show that $M$ is $\lambda_0$-presentable. Let $g: M \rightarrow N$, where $N$ is the directed colimit of a $\lambda_0$-directed system with diagram maps $\psi_{\beta \alpha} : N_\alpha \rightarrow N_\beta$ and colimit maps $\psi_{\beta} : N_\beta \rightarrow N$. Since $M_\alpha$ is $\lambda_0$-presentable, the map $g \phi_\alpha$ must factor essentially uniquely through some $N_\beta$, i.e.\ there exists an essentially unique $g_{\beta \alpha} : M_\alpha \rightarrow N_\beta$ so that $\psi_\beta g_{\beta \alpha} = g \phi_\alpha$. Now let $g_\beta := g_{\beta \alpha} f_\alpha$. Then $\psi_\beta g_\beta = \psi_\beta g_{\beta \alpha} f_\alpha = g \phi_\alpha f_\alpha = g \id_M = g$, so $g$ factors through $N_\beta$. Let us now see that $g_\beta$ is essentially unique. Suppose that $g = \psi_{\beta'} h_{\beta'}$ for some $h_{\beta'} : M \rightarrow N_{\beta'}$. Let $h_{\beta' \alpha} := h_\beta \phi_\alpha$. By essential uniqueness of $g_{\beta \alpha}$, we know that there is $\gamma$ above both $\beta$ and $\beta'$ such that $\psi_{\gamma \beta} g_{\beta \alpha} = \psi_{\gamma \beta'} h_{\beta' \alpha}$. We then have that $\psi_{\gamma \beta} g_\beta = \psi_{\gamma \beta} g_{\beta \alpha} f_\alpha = \psi_{\gamma \beta'} h_{\beta' \alpha} f_\alpha = \psi_{\gamma \beta'} h_{\beta'} \phi_{\alpha} f_\alpha = \psi_{\gamma \beta'} h_{\beta'}$, as desired.
  \end{enumerate}
\end{proof}

We can give the following bounds on the presentability of the colimit of a $(\mu, <\lambda)$-system:

\begin{lem}\label{system-rank}
  Let $\mu, \lambda$ be cardinals with $\mu$ regular. Let $\ck$ be a category with $\mu$-directed colimits. Let $\BI$ be a $(\mu, <\lambda)$-system with diagram maps $\phi_{\beta \alpha} : M_\alpha \rightarrow M_\beta$, $\alpha, \beta \in I$, and let $M$ be its colimit. Suppose that for $\alpha \in I$, $M_\alpha$ is $\lambda_\alpha$-presentable, $\lambda_\alpha < \lambda + \aleph_1$. Then:

  \begin{enumerate}
  \item If $\BI$ is proper, then $M$ is not $\mu$-presentable.
  \item $M$ is $\left(|I|^+ + \sup_{\alpha \in I} \lambda_\alpha\right)_r$-presentable. In particular, $M$ is $(|I|^+ + \lambda_r)$-presentable.
  \item If $\cf{\lambda} > |I|$ and $\lambda$ is not the successor of a singular cardinal, then $M$ is $(<(|I|^{++} + \lambda))$-presentable.
  \end{enumerate}
\end{lem}
\begin{proof}
  The first part is by Lemma \ref{system-resolvable}. For the second part, let $\lambda' := \left(|I|^+ + \sup_{\alpha \in I} \lambda_\alpha\right)_r$.
  
  Pick a $\lambda'$-directed system $\BJ$ with objects $\seq{N_\gamma : \gamma \in J}$ and with colimit $N$. Let $g : M \rightarrow N$ be a morphism. For each $\alpha \in I$, $M_\alpha$ is $\lambda_\alpha$-presentable, hence $\lambda'$-presentable, $g$ factors essentially uniquely through some $N_{\gamma_\alpha}$. Let $\gamma$ be an upper bound to all the $\gamma_\alpha$'s (exists since $|I| < \lambda'$ and $J$ is $\lambda'$-directed). Then $g$ factors essentially uniquely through $N_\gamma$, as desired.

  The third part follows directly from the second.
\end{proof}

Using the terminology above, a $\lambda$-accessible category is a category which has $\lambda$-directed colimits, is $(\lambda, \lambda)$-resolvable, and has a set of $\lambda$-presentable objects, up to isomorphism. We will consider the following parameterized generalization:

\begin{defin}
  Let $\kappa \le \mu \le \lambda$ be cardinals with $\kappa$ and $\mu$ regular. A category $\ck$ is \emph{$(\kappa, \mu, <\lambda)$-accessible} if:

  \begin{enumerate}
  \item It has $\kappa$-directed colimits.
  \item It is $(\mu, <\lambda)$-resolvable.
  \item It has only a set (up to isomorphism) of $(<\lambda)$-presentable objects.
  \end{enumerate}

  Define \emph{$(\kappa, \mu, \lambda)$-accessible} similarly. We say that $\ck$ is \emph{$(\mu, <\lambda)$-accessible} if it is $(\mu, \mu, <\lambda)$-accessible.  As noted above, a category $\ck$ is \emph{$\mu$-accessible} just in case it is $(\mu, <\mu^+)$-accessible. We say that $\ck$ is \emph{accessible} if it is $\mu$-accessible for some regular cardinal $\mu$.
\end{defin}
\begin{remark}
  Assume that $\ck$ is $(\kappa, \mu, <\lambda)$-accessible. We have the following monotonicity properties:

  \begin{enumerate}
  \item If $\kappa' \in [\kappa, \mu]$ is regular, then $\ck$ is $(\kappa', \mu, <\lambda)$-accessible.
  \item If $\mu' \in [\kappa, \mu]$ is regular, then $\ck$ is $(\kappa, \mu', <\lambda)$-accessible.
  \item If $\lambda' \ge \lambda$, then $\ck$ is $(\kappa, \mu, <\lambda')$-accessible.
  \end{enumerate}
\end{remark}

The proof of \cite[2.3.10]{makkai-pare} gives the following way of raising the index $\mu$ in the definition of a $(\kappa, \mu, <\lambda)$-accessible category. We sketch a proof for the convenience of the reader:

\begin{fact}\label{triangle-acc}
  Let $\kappa < \mu \le \lambda$ be cardinals with $\kappa$ and $\mu$ regular and $\cf{\lambda} \ge \mu$. Let $\ck$ be a category with $\kappa$-directed colimits. If $M \in \ck$ is $(\kappa, <\lambda)$-resolvable and $\kappa \tlt \mu$ (see Definition \ref{triangle-def}), then $M$ is the colimit of a $\mu$-directed system where each object is a $\kappa$-directed colimit of strictly less than $\mu$-many $(<\lambda)$-presentable objects. In particular, $M$ is $(\mu, <(\lambda + \mu^\ast))$-resolvable (see Definition \ref{star-def}).
\end{fact}
\begin{proof}[Proof sketch]
  Suppose that the $(\kappa, <\lambda)$-system is indexed by $I$ and has colimit $M$. Since $\kappa \tleq \mu$, any subset of $I$ of cardinality strictly less than $\mu$ is contained inside a $\kappa$-directed subset of $I$ of cardinality strictly less than $\mu$. Thus the set of all $\kappa$-directed subsets of $I$ of cardinality strictly less than $\mu$ is $\mu$-directed. It is easy to check that it induces the desired $(\mu, <(\lambda + \mu^\ast))$-system whose colimit is $M$.
\end{proof}

Note that this implies in particular that any $(\kappa, \mu, <\lambda)$-accessible category is $\theta$-accessible for some $\theta$ (for example, $\theta = \left(2^{\lambda^+}\right)^+$ suffices).

By Lemma \ref{system-rank}, any object of an accessible category is $\lambda$-presentable for some regular cardinal $\lambda$. Therefore it makes sense to define:

\begin{defin} Let $\ck$ be an accessible category.
\begin{enumerate}
\item For any $M\in\ck$, the {\it presentability rank of $M$}, denoted $r_{\ck} (M)$, is the least regular cardinal $\lambda$ such that $M$ is $\lambda$-presentable.
\item We denote by $|M|_{\ck}$ the {\it internal size of $M$}. It is defined by:

  $$
  |M|_{\ck} = \begin{cases}
    \mu & \text{ if } r_{\ck} (M) = \mu^+ \\
    r_{\ck} (M) & \text{ if } r_{\ck} (M) \text{ is limit}
  \end{cases}
  $$
  
\end{enumerate}
\end{defin}

The following gives a criteria for when the presentability rank is a successor:

\begin{thm}\label{successor-thm}
  Let $\lambda$ be a regular cardinal and let $\ck$ be a $(\lambda, <\lambda)$-accessible category. If $M \in \ck$ is $\lambda$-presentable, then $M$ is $(<\lambda)$-presentable.
\end{thm}
\begin{proof}
  Otherwise, $M$ would be $(\lambda, <\lambda)$-resolvable (by definition of accessibility) but not $(<\lambda)$-presentable. Therefore there is a proper $(\lambda, <\lambda)$-system whose colimit is $M$ by Lemma \ref{system-resolvable}. By Lemma \ref{system-rank}, $r_{\ck} (M) > \lambda$, contradicting $\lambda$-presentability.
\end{proof}

Note that $(\aleph_0, <\aleph_0)$-accessible is the same as $\aleph_0$-accessible. In addition, if a category is $(\lambda^+, <\lambda^+)$-accessible then it is quite pathological (for example Theorem \ref{successor-thm} says that it cannot have an object of presentability rank $\lambda^+$). Thus we will use Theorem \ref{successor-thm} only when $\lambda$ is weakly inaccessible and uncountable. 

Assuming that every weakly inaccessible cardinal is sufficiently closed, we obtain that the presentability rank is always a successor. This weakens the GCH hypothesis in \cite[2.3(5)]{beke-rosicky} and also shows (Remark \ref{sch-rmk}) that sufficiently large presentability ranks are always successors if there is an $\omega_1$-strongly compact cardinal. Note that we could have obtained the result directly by carefully examining the proof in \cite{beke-rosicky}, but Theorem \ref{successor-thm} is new and can be used even in situations where SCH does not hold (see Corollary \ref{inter-succ}).

\begin{cor}
  Let $\ck$ be a $(\mu, <\lambda)$-accessible category and let $\theta \ge \lambda$ be weakly inaccessible. If $\theta$ is $\mu$-closed, then the presentability rank of an object in $\ck$ cannot be equal to $\theta$.
\end{cor}
\begin{proof}
  By Fact \ref{triangle-fact}, $\mu \tlt \theta$. Note that $\ck$ is in particular $(\mu, <\theta)$-accessible so by Fact \ref{triangle-acc} also $(\theta, <\theta)$-accessible. Now apply Theorem \ref{successor-thm}.
\end{proof}

Note that a weakly inaccessible $\theta > \mu$ is $\mu$-closed if and only if $\SCH_{\mu, \theta}$ holds (see Lemma \ref{mu-closed-sch}). This is quite a weak assumption, although its failure is still consistent (assuming quite large cardinals, see Remark \ref{sch-rmk}). Note also that when $\mu = \aleph_0$, $\theta$ is always $\mu$-closed so we recover \cite[4.2]{beke-rosicky}: sufficiently large presentability ranks are successors in any accessible category with directed colimits. Indeed, a $\lambda$-accessible category with directed colimits is, in particular, an $(\aleph_0, <\lambda^+)$-accessible category.

We now move to the study of accessible categories whose morphisms are monomorphisms. Recall (Fact \ref{mu-aec-acc}) that these are the same as $\mu$-AECs. The following gives a criteria for when an object of a certain presentability rank exists:

\begin{thm}\label{system-existence-thm}
  Let $\mu < \lambda$ be cardinals with $\mu$ regular, $\cf{\lambda} > \mu$, and $\lambda$ not the successor of a singular cardinal. Let $\ck$ be a category with $\mu$-directed colimits and all morphisms monomorphisms. 

  \begin{enumerate}
  \item If $\ck$ has a proper $(\mu, <\lambda)$-system, then there exists an object $M \in \ck$ such that $\mu < r_{\ck} (M) < \lambda + \mu^{++}$.
  \item If $\ck$ is $(\mu, <\lambda)$-accessible and has an object that is not $\mu$-presentable, then there exists $M \in \ck$ with $\mu < r_{\ck} (M) < \lambda + \mu^{++}$.
  \end{enumerate}
\end{thm}
\begin{proof} \
  \begin{enumerate}
  \item Fix a proper $(\mu, <\lambda)$-system. Using that all morphisms are monomorphisms, one can fix a subsystem of it indexed by a chain of length $\mu$ which is also proper. Now its colimit is as desired by Lemma \ref{system-rank}.
  \item Pick $M \in \ck$ that is not $\mu$-presentable. By definition of accessibility, $M$ is $(\mu, <\lambda)$-resolvable. If $M$ is $(<\lambda)$-presentable, we are done so assume that $M$ is not $(<\lambda)$-presentable. By Lemma \ref{system-resolvable}, there is a proper $(\mu, <\lambda)$-system whose colimit is $M$. By the previous part, $\ck$ has an object of presentability rank strictly between $\mu$ and $\lambda + \mu^{++}$, as desired.
  \end{enumerate}
\end{proof}

Note the following special case:

\begin{cor}\label{mono-special-case}
  Let $\ck$ be a $(\lambda, \lambda^+)$-accessible category with all morphisms monomorphisms. If $\ck$ has an object that is not $\lambda$-presentable, then $\ck$ has an object of presentability rank $\lambda^+$.
\end{cor}
\begin{proof}
  Apply the second part of Theorem \ref{system-existence-thm} with $(\lambda, \lambda^{++})$ in place of $(\mu, \lambda)$.
\end{proof}

We deduce that, when all morphisms are monomorphisms and the category is large, the accessibility spectrum (the set of cardinals $\lambda$ such that the category is $\lambda$-accessible) is contained in the existence spectrum. In particular, if such a category is well-accessible (in the sense that its accessibility spectrum is a tail, see \cite[2.1]{beke-rosicky}) then it is weakly LS-accessible: it has an object of every \emph{regular} internal size (Definition \ref{ls-acc-def}).

\begin{cor}
  A large $\lambda$-accessible category all of whose morphisms are monomorphisms has an object of internal size $\lambda$. 
\end{cor}
\begin{proof}
  Immediate from Corollary \ref{mono-special-case}.
\end{proof}

We deduce a general result on the existence spectrum of a large accessible category whose morphisms are monomorphisms.

\begin{cor}\label{mono-internal-size}
  Let $\mu \le \lambda$ both be regular cardinals and let $\ck$ be a $(\mu, \lambda^+)$-accessible category with all morphisms monomorphisms. If $\ck$ has an object that is not $\lambda$-presentable, then there is $M \in \ck$ such that $\lambda \le |M|_{\ck} \le \lambda^{<\mu}$.
\end{cor}
\begin{proof}
  Let $\theta := \left(\lambda^{<\mu}\right)^+$. Note that $\theta$ is $\mu$-closed, so by Fact \ref{triangle-fact}, $\mu \tlt \theta$. By Fact \ref{triangle-acc} , $\ck$ is $(\mu, \theta, \theta)$-accessible, hence $(\lambda, \theta)$-accessible. By the second part of Theorem \ref{system-existence-thm} (where $(\mu, \lambda)$ there stands for $(\lambda, \theta^+)$ here), there exists $M \in \ck$ with $\lambda < r_{\ck} (M) < \theta^+ + \lambda^{++} = \theta^+$. Thus $\lambda \le |M|_{\ck} < \theta$, i.e.\ $\lambda \le |M|_{\ck} \le \lambda^{<\mu}$, as desired.
\end{proof}

We will see later (Corollary \ref{internal-size-cor-2}) that the assumption of regularity of $\lambda$ can be relaxed assuming SCH.

\section{Presentability in $\mu$-AECs}\label{mu-aecs-sec}

Recall from \cite[\S2]{mu-aec-jpaa} that a \emph{($\mu$-ary) abstract class} is a pair $\K = (K, \lea)$ such that $K$ is a class of structures is a fixed $\mu$-ary vocabulary $\tau = \tau (\K)$, and $\lea$ is a partial order on $K$ that respects isomorphisms and extends the $\tau$-substructure relation. In any abstract class $\K$, there is a natural notion of morphism: we say that $f: M \rightarrow N$ is a \emph{$\K$-embedding} if $f$ is an isomorphism from $f$ onto $f[M]$ and $f[M] \lea N$. We see an abstract class and its $\K$-embeddings as a category (but we still use boldface, i.e.\ denote it by $\K$ and not $\ck$, to emphasize the concreteness of the category). In fact (see \cite[\S2]{mu-aec-jpaa}), an abstract class is a replete and iso-full subcategory of the category of $\tau$-structures and $\tau$-structure embeddings.  It is moreover a \emph{concrete} category: we will write $U M$ for the universe of a member $M$ of $\K$. 

We now recall the definition of a $\mu$-AEC from \cite[2.2]{mu-aec-jpaa}:

\begin{defin}
  Let $\mu$ be a regular cardinal. An abstract class $\K$ is a \emph{$\mu$-abstract elementary class} (or \emph{$\mu$-AEC} for short) if it satisfies the following three axioms:

  \begin{enumerate}
  \item Coherence: for any $M_0, M_1, M_2 \in \K$, if $M_0 \subseteq M_1 \lea M_2$ and $M_0 \lea M_2$, then $M_0 \lea M_1$.
  \item Chain axioms: if $\seq{M_i : i \in I}$ is a $\mu$-directed system in $\K$, then:
    \begin{enumerate}
    \item $M := \bigcup_{i \in I} M_i$ is in $\K$.
    \item $M_i \lea M$ for all $i \in I$.
    \item If $M_i \lea N$ for all $i \in I$, then $M \lea N$.
    \end{enumerate}
  \item Löwenheim-Skolem-Tarski (LST) axiom: there exists a cardinal $\lambda = \lambda^{<\mu} \ge |\tau (\K)| + \mu$ such that for any $M \in \K$ and any $A \subseteq U M$, there exists $M_0 \in \K$ with $M_0 \lea M$, $A \subseteq U M_0$, and $|U M_0| \le |A|^{<\mu} + \lambda$. We write $\LS (\K)$ for the least such $\lambda$.
  \end{enumerate}
\end{defin}

Note that when $\mu = \aleph_0$, we recover Shelah's definition of an AEC from \cite{sh88}. The connection of $\mu$-AECs to accessible categories is given by \cite[\S 4]{mu-aec-jpaa}:

\begin{fact}\label{mu-aec-acc}
  If $\K$ is a $\mu$-AEC, then it is a $(\mu, \LS (\K)^+)$-accessible category whose morphisms are monomorphisms. Conversely, any $\mu$-accessible category whose morphisms are monomorphisms is equivalent to a $\mu$-AEC.
\end{fact}

Thus applying Corollary \ref{mono-internal-size}, we immediately get that for any $\mu$-AEC $\K$ with arbitrarily large models and any $\lambda \ge \LS (\K)$ regular, there exists $M \in \K$ with $\lambda \le |M|_{\K} \le \lambda^{<\mu}$. The aim of this section is to investigate the relationship between internal sizes and cardinalities in $\mu$-AECs. The main result is that assuming GCH, or weakening of the form described in Definition~\ref{sch-def}, internal size and cardinality agree on any sufficiently large model whose cardinality is not the successor of a cardinal of cofinality less than $\mu$. From this we can conclude further results on the existence spectrum.

First note that the definition of presentability simplifies in $\mu$-AECs:

\begin{lem}\label{mu-aec-present}
  Let $\K$ be a $\mu$-AEC, let $\lambda \ge \mu$ be a regular cardinal, and let $M \in \K$. Then $M$ is $\lambda$-presentable if and only if for any $\lambda$-directed system $\seq{M_i : i \in I}$, if $M \lea \bigcup_{i \in I} M_i$, then there exists $i \in I$ such that $M \lea M_i$.
\end{lem}
\begin{proof}
($\Rightarrow$) Let $M$ be $\lambda$-presentable, and let $\seq{M_i : i \in I}$ be a $\lambda$-directed system, with $M \lea \bigcup_{i \in I} M_i$. By $\lambda$-presentability, there exists $i \in I$ such that $M \lea M_i$. 

($\Leftarrow$) Say we have $M\to N$, $N=\mbox{colim}_{i\in I} N_i$ with $I$ $\lambda$-directed.  Then $N$ is a $\lambda$-directed union of the images of the $N_i$ under the colimit coprojections, and so by hypothesis, the image of $M$ in $N$ must land in one of them. Coherence does the rest.
\end{proof}

Toward bounding how big the presentability of an object can be, we look at what it means for a model to be minimal over a set:

\begin{defin}\label{minimal-def}
  Let $\K$ be a $\mu$-AEC, let $M \in \K$, and let $A \subseteq U M$. We say that $M$ is \emph{minimal over $A$} if for any $M_0, N \in \K$, if $M \lea N$, $M_0 \lea N$, and $A \subseteq U M_0$, then $M \lea M_0$.
\end{defin}

\begin{lem}\label{pres-1}
  Let $\K$ be a $\mu$-AEC and let $\lambda \ge \mu$ be a regular cardinal. Let $M \in \K$ and let $A \subseteq U M$ be such that $M$ is minimal over $A$. If $|A| < \lambda$, then $M$ is $\lambda$-presentable.
  
  In particular (taking $A := U M$), if $|U M| < \lambda$, then $M$ is $\lambda$-presentable.
\end{lem}
\begin{proof}
  We use Lemma \ref{mu-aec-present}. Let $\seq{M_i : i \in I}$ be a $\lambda$-directed system such that $M \lea \bigcup_{i \in I} M_i$. Since the system is $\lambda$-directed, there exists $i \in I$ such that $A \subseteq U M_i$. By definition of minimality, this means that $M \lea M_i$, as desired.
\end{proof}

\begin{remark}
  The assumption that $\lambda \ge \mu$ cannot in general be removed, see Example \ref{metric-example}.
\end{remark}

Assuming a certain closure condition, we obtain a converse to Lemma \ref{pres-1}:

\begin{defin}\label{closed-def}
  Let $\K$ be a $\mu$-AEC, let $\lambda$ be a regular cardinal, and let $M \in \K$. $M$ is \emph{$\lambda$-closed} if for any $A \subseteq U M$ of cardinality less than $\lambda$ there exists $M_0 \in \K$ with $M_0 \lea M$ of cardinality less than $\lambda$ and containing $A$.
\end{defin}

\begin{lem}\label{pres-2}
  Let $\K$ be a $\mu$-AEC and let $\lambda \ge \mu$ be a regular cardinal. If $M$ is $\lambda$-closed and $\lambda$-presentable, then $| U M| < \lambda$.
\end{lem}
\begin{proof}
  Let $S := \{M_0 \lea M \mid |U M_0| < \lambda\}$. Then $S$ is $\lambda$-directed and (because $M$ is $\lambda$-closed), $M$ is the colimit of the system induced by $S$. Since $M$ is $\lambda$-presentable, there exists $M_0 \in S$ such that $M \lea M_0$, hence $M$ has size less than $\lambda$.
\end{proof}

Note that if $\lambda$ is a sufficiently-nice cardinal, then any member of $\K$ will be $\lambda$-closed:

\begin{lem}\label{closed-lst}
  Let $\K$ be a $\mu$-AEC and let $\lambda > \LS (\K)$. If $\lambda$ is $\mu$-closed and $\lambda > \LS (\K)$, then any element of $\K$ is $\lambda$-closed.
\end{lem}
\begin{proof}
  By the Löwenheim-Skolem-Tarski axiom of $\mu$-AECs.
\end{proof}

By \cite[4.2]{mu-aec-jpaa}, the behavior of $|M|_{\K}$ around $\LS (\K)$ is well-understood. We give a proof here again for completeness:

\begin{fact}\label{internal-around-ls}
  Let $\K$ be a $\mu$-AEC and let $M \in \K$. The following are equivalent:

  \begin{enumerate}
  \item\label{around-ls-1} $|U M| \le \LS (\K)$.
  \item\label{around-ls-2} $M$ is $\LS (\K)^+$-presentable.
  \item\label{around-ls-3} $|M|_{\K} \le \LS (\K)$.
  \end{enumerate}
\end{fact}
\begin{proof}
  By definition of internal size, (\ref{around-ls-2}) is equivalent to (\ref{around-ls-3}). Moreover, Lemma \ref{pres-1} says in particular that (\ref{around-ls-1}) implies (\ref{around-ls-3}). Conversely, assume that $M$ is $\LS (\K)^+$-presentable. The axioms of $\mu$-AECs imply that $\LS (\K)^{<\mu} = \LS (\K)$. Therefore by Lemma \ref{closed-lst} (where $\lambda$ there stands for $\LS (\K)^+$ here), $M$ is $\LS (\K)^+$-closed. By Lemma \ref{pres-2}, $|U M| < \LS (\K)^+$. Therefore (\ref{around-ls-2}) implies (\ref{around-ls-1}).
\end{proof}

We now attempt to compute $|M|_{\K}$ when $|U M| > \LS (\K)$.

\begin{lem}\label{zfc-size}
  Let $\K$ be a $\mu$-AEC and let $M \in \K$. Let $\lambda := |U M|$ and assume that $\lambda > \LS (\K)$. Then:

  \begin{enumerate}
  \item $|M|_{\K} \le \lambda$ (note: this also holds when $\lambda^+ \in [\mu, \LS (\K)^+]$).
  \item If $\lambda_0 \le \lambda$ is regular such that $M$ is $\lambda_0$-closed, then $r_{\K} (M) > \lambda_0$, so in particular $|M|_{\K} \ge \lambda_0$.
  \item If $\lambda$ is $\mu$-closed, then $r_{\K} (M) = \lambda^+$, so $|M|_{\K} = \lambda$.
  \end{enumerate}
\end{lem}
\begin{proof} \
  \begin{enumerate}
  \item By Lemma \ref{pres-1}, $M$ is always $\lambda^+$-presentable.
  \item We know that $M$ is not $\LS (\K)^+$-presentable (since we are assuming $\lambda > \LS (\K)$). Thus we may assume without loss of generality that $\lambda_0 > \LS (\K)^+$. By Lemma \ref{pres-2}, $M$ is not $\lambda_0$-presentable, hence $r_{\K} (M) > \lambda_0$, so $|M|_{\K} \ge \lambda_0$.
  \item If $\lambda$ is regular, then by the previous part used with $\lambda_0 = \lambda$, $r_{\K} (M) \ge \lambda^+$ and this is an equality by the first part. If $\lambda$ is singular (hence limit), then by Lemma \ref{closed-lst} and the previous parts, it suffices to show that there are unboundedly-many regular cardinals $\lambda_0 \le \lambda$ that are $\mu$-closed. Let $\theta < \lambda$ and let $\lambda_0 := (\theta^{<\mu})^+$. Then $\lambda_0 < \lambda$ (as $\lambda$ is $\mu$-closed and limit), $\lambda_0$ is regular, and $\lambda_0$ is $\mu$-closed, as needed.
  \end{enumerate}
\end{proof}

We obtain the following inequality:

\begin{thm}\label{size-ineq}
  Let $\K$ be a $\mu$-AEC and let $M \in \K$. Then:

  $$
  |M|_{\K} \le |U M| + \LS (\K) \le |M|_{\K}^{<\mu} + \LS (\K)
  $$
  
\end{thm}
\begin{proof}
  By Fact \ref{internal-around-ls}, $|U M| \le \LS (\K)$ if and only if $M$ is $\LS (\K)^+$-presentable. This together with the first part of Lemma \ref{zfc-size} gives the first inequality. For the second, assume for a contradiction that $|U M| + \LS (\K) > \left(|M|_{\K} + \LS (\K)\right)^{<\mu}$. In particular, $|U M| > \LS (\K)$. Let $\lambda := |U M|$ and let $\lambda_0 := \left(\left(|M|_{\K} + \LS (\K)\right)^{<\mu}\right)^+$. Then $\lambda_0 \le \lambda$, $\lambda_0$ is regular, and $\lambda_0$ is $\mu$-closed. By Lemma \ref{closed-lst} and the second part of Lemma \ref{zfc-size}, $|M|_{\K} \ge \lambda_0$. This contradicts the definition of $\lambda_0$.
\end{proof}

This gives a less desirable relationship between internal size and cardinality than one might like:

\begin{thm}\label{internal-size-mu-aec}
  Let $\K$ be a $\mu$-AEC, let $M \in \K$, and let $\lambda := |U M|$. Assume that $\lambda > \LS (\K)$ and assume GCH (or just $\SCH_{\mu, \lambda}$, see Definition \ref{sch-def}). 
  
  \begin{enumerate}
  \item If $\lambda$ is not the successor of a cardinal of cofinality less than $\mu$, then $r_{\K} (M) = \lambda^+$ so $|M|_{\K} = \lambda$.
  \item If $\lambda = \lambda_0^+$ for some $\lambda_0$ with $\cf{\lambda_0} < \mu$, then $|M|_{\K}$ is either $\lambda_0$ or $\lambda$. 
  \end{enumerate}
\end{thm}
\begin{proof} \
  \begin{enumerate}
  \item By Lemma \ref{zfc-size}, it is enough to show that $\lambda$ is $\mu$-closed. This follows from Lemma \ref{mu-closed-sch}.
  \item Since $\SCH_{\mu, \lambda}$ holds, we must have that $\lambda_0$ is almost $\mu$-closed. Since $\cf{\lambda_0} < \mu$, we cannot have that $\lambda_0^{<\mu} = \lambda_0$, so we must have that $\lambda_0$ is $\mu$-closed. Since $\lambda_0$ is limit, there must be unboundedly-many regular $\lambda_0' < \lambda_0$ that are $\mu$-closed (namely the cardinals of the form $\left(\theta^{<\mu}\right)^+$ for $\theta < \lambda_0$), and hence so that (Lemma \ref{closed-lst}) $M$ is $\lambda_0'$-closed. By Lemma \ref{zfc-size}, $\lambda_0 \le |M|_{\K} \le \lambda$. 
  \end{enumerate}
\end{proof}

Assuming GCH (or merely a suitable instance of SCH), we can use this to relax the regularity assumption on $\lambda$ in Corollary \ref{mono-internal-size}:

\begin{cor}\label{internal-size-cor}
  Let $\K$ be a $\mu$-AEC and let $\lambda > \LS (\K)$. Assume GCH (or just $\SCH_{\mu, \lambda}$). If $\lambda = \lambda^{<\mu}$ and $\K$ has a model of cardinality at least $\lambda^+$, then there exists $M \in \K$ such that $|U M| = |M|_{\K} = \lambda$.
\end{cor}
\begin{proof}
  If $\lambda$ is regular, then Corollary \ref{mono-internal-size} gives the result (recalling Fact \ref{mu-aec-acc}), so assume that $\lambda$ is singular. By the Löwenheim-Skolem-Tarski axiom, there exists a model $M \in \K$ of cardinality $\lambda$. By Theorem \ref{internal-size-mu-aec} (note that $\lambda$ is not a successor, so the first part there must apply), we get that $|M|_{\K} = \lambda$.
\end{proof}

\begin{cor}\label{internal-size-cor-2}
  Let $\ck$ be a large $\mu$-accessible category with all morphisms monomorphisms. Assuming GCH (or just $\SCH_{\mu, \ge \theta}$ for some $\theta)$, $\ck$ has an object of all high-enough internal sizes of cofinality at least $\mu$.
\end{cor}
\begin{proof}
  By Fact \ref{mu-aec-acc}, $\ck$ is equivalent to a $\mu$-AEC $\K$. Fix $\theta > \LS (\K)$ such that $\SCH_{\mu, \ge \theta}$ and let $\lambda > \theta$ have cofinality at least $\mu$. By Lemma \ref{sch-lem}, $\lambda = \lambda^{<\mu}$. Now apply Corollary \ref{internal-size-cor}.
\end{proof}

We can also give a condition under which there is \emph{no} model of a given internal size (although we do not know whether it can ever hold):

\begin{thm}\label{not-ls-acc-cond}
  Let $\K$ be a $\mu$-AEC and let $\lambda > \LS (\K)$ be such that $\lambda < \lambda^{<\mu}$. If $\K$ has no model with cardinality in $[\lambda, \lambda^{<\mu})$ and $\K$ is categorical in cardinality $\lambda^{<\mu}$, then $\K$ has no model of internal size $\lambda$.
\end{thm}
\begin{proof}
  Assume for a contradiction that there is $M \in \K$ with $|M|_{\K} = \lambda$. By Theorem \ref{size-ineq}, $\lambda \le |U M| \le \lambda^{<\mu}$, and since there are no models of cardinality in $[\lambda, \lambda^{<\mu})$, we must have that $|U M| = \lambda^{<\mu}$. By Corollary \ref{mono-internal-size} (where $\lambda$ there stands for $\lambda^+$ here), there is $N \in \K$ with $\lambda^+ \le |N|_{\K} \le \left(\lambda^+\right)^{<\mu} = \lambda^{<\mu}$. Again since there are no models in $[\lambda, \lambda^{<\mu})$ we must have that $|U N| = \lambda^{<\mu}$. By construction $M$ and $N$ are not isomorphic, contradicting categoricity in $\lambda^{<\mu}$.
\end{proof}

\begin{question}
  Does there exist a $\mu$-AEC $\K$ such that for any big-enough cardinal $\lambda$ with $\lambda < \lambda^{<\mu}$, $\K$ has no model in $[\lambda, \lambda^{<\mu})$ but is categorical in $\lambda^{<\mu}$?
\end{question}

By Theorem \ref{not-ls-acc-cond}, such an example cannot be LS-accessible.

\section{$\mu$-AECs admitting intersections and LS-accessibility}\label{intersections-sec}

We recall the definition of a $\mu$-AEC admitting intersections.  For $\mu = \aleph_0$, the definition first appears in \cite[1.2]{non-locality} and is studied in \cite[\S2]{ap-universal-apal}. The definition for uncountable $\mu$ is introduced in \cite{multipres-pams}. We note, in connection with the appendix, that a $\mu$-AEC admitting intersections is the same, up to equivalence of categories, as a locally $\mu$-polypresentable category all of whose morphisms are monomorphisms (see \cite[5.7]{multipres-pams}):

\begin{defin}
  A $\mu$-AEC $\K$ \emph{admits intersections} if for any $N \in \K$ and any $A \subseteq U N$, the set

  $$
  \cclp{\K}^N (A) = \ccl^{N} (A) := \bigcap \{M \in \K \mid M \lea N, A \subseteq U M\}
  $$
is the universe of a $\lea$-substructure of $M$. In this case, we abuse notation and write $\ccl^{N} (A)$ for this substructure as well.
\end{defin}
\begin{remark}[2.14(3) in \cite{ap-universal-apal}]
  If $\K$ is a $\mu$-AEC admitting intersections, $M \lea N$ and $A \subseteq U M$, then $\ccl^M (A) = \ccl^N (A)$. This will be used without comment.
\end{remark}

%% The following is proven for $\mu = \aleph_0$ in \cite[2.14(6)]{ap-universal-apal}. The proof generalizes to uncountable $\mu$ without difficulties.

%% \begin{fact}\label{mu-charact-inter}
%%   Let $\K$ be a $\mu$-AEC admitting intersections. For any $M \in \K$ and any $A \subseteq B \subseteq U M$, if $A \subseteq \cl^M (B)$ and $|A| < \mu$, then there exists $B_0 \subseteq B$ such that $|B_0| < \mu$ and $A \subseteq \cl^M (B_0)$.
%% \end{fact}

Crucially, in a $\mu$-AEC admitting intersections, the closure of a set is minimal over the set (in the sense of Definition \ref{minimal-def}):

\begin{lem}\label{admit-intersec-min}
  Let $\K$ be a $\mu$-AEC admitting intersections. Let $M \in \K$ and let $A \subseteq U M$. If $M = \cl^M (A)$, then $M$ is minimal over $A$.
\end{lem}
\begin{proof}
  Let $M_0, N \in \K$ be such that $M \lea N$, $M_0 \lea N$, and $A \subseteq U M_0$. Then:

      $$
      M = \cl^{M} (A) = \cl^{N} (A) = \cl^{M_0} (A)
      $$
      
      Thus $M \lea M_0$, as desired.  
\end{proof}

We get that $\mu$-AECs admitting intersections behave very well with respect to the accessibility rank. They are $\lambda$-accessible for every regular $\lambda \ge \mu$, and in fact even $(\lambda, <\lambda)$-accessible when $\lambda$ is weakly inaccessible. This generalizes \cite[3.3]{multipres-pams}.

\begin{thm}\label{accessibility-rank-inter}
  Let $\mu \le \lambda$ both be regular cardinals. If $\K$ is a $\mu$-AEC admitting intersections, then $\K$ is $(\mu, \lambda, <\lambda^\ast)$-accessible (see Definition \ref{star-def}).
\end{thm}
\begin{proof}
  By definition, $\K$ has $\mu$-directed colimits and a set of $(<\lambda^\ast)$-presentable objects. It remains to see that it is $(\lambda, <\lambda^\ast)$-resolvable. Let $M \in \K$. Consider the set $I := \{A \subseteq U M \mid |A| < \lambda^\ast\}$. For $A \in I$, let $M_{A} := \cl^M (A)$. Note that $M_{A} \lea M$ since $\K$ admits intersections. Moreover, $M_A$ is $|A|^+$-presentable by Lemmas \ref{admit-intersec-min} and \ref{pres-1}. In particular, $M_A$ is $(<\lambda^\ast)$-presentable. Therefore $\seq{M_{A} : A \in I}$ is a $(\lambda, <\lambda^\ast)$-system whose colimit is $M$, as desired.
\end{proof}

Applying the results of Section \ref{foundations-sec}, we immediately obtain:

\begin{cor}\label{inter-succ}
  Let $\K$ be a $\mu$-AEC admitting intersections.
  
  \begin{enumerate}
  \item In $\K$, presentability ranks that are greater than or equal to $\mu$ are successors.
  \item If $\K$ has arbitrarily large models, then $\K$ has objects of all regular internal sizes greater than or equal to $\mu$. In particular, it is weakly LS-accessible.
  \end{enumerate}
\end{cor}
\begin{proof} \
  \begin{enumerate}
  \item By Theorem \ref{accessibility-rank-inter} and Theorem \ref{successor-thm}.
  \item By Theorem \ref{accessibility-rank-inter} and Corollary \ref{mono-special-case}.
  \end{enumerate}
\end{proof}

We conclude that in a $\mu$-AEC admitting intersections there is a natural way of computing the internal size:

\begin{defin}\label{norm-cl}
  Let $\K$ be a $\mu$-AEC admitting intersections. For any $M \in \K$, define:

  $$
  |M|_{\cl} := \min\{|A| \mid A \subseteq UM, M = \cl^M (A)\}
  $$
  
\end{defin}

\begin{thm}\label{internal-cl}
  Let $\K$ be a $\mu$-AEC admitting intersections. For any $M \in \K$, $r_{\K} (M) + \mu = |M|_{\cl}^+ + \mu$. In particular if $|U M| > \LS (\K)$ then $r_{\K} (M) = |M|_{\cl}^+$ so $|M|_{\K} = |M|_{\cl}$.
\end{thm}
\begin{proof}
  Fix $A$ realizing the minimum in the definition of $|M|_{\cl}$. By Lemma \ref{admit-intersec-min}, $M$ is minimal over $A$. Therefore by Lemma \ref{pres-1}, $M$ is $(|A|^+ + \mu)$-presentable, so $r_{\K} (M) \le |A|^+ + \mu$. 

  We now show that $|A|^+ + \mu \le r_{\K} (M) + \mu$. If $|A| < \mu$, then by what has just been established $r_{\K} (M) \le \mu$, so we are done. Assume now that $|A| \ge \mu$. Let $\lambda := |A|$ and let $\lambda_0 \in [\mu, \lambda]$ be a regular cardinal. We show that $M$ is not $\lambda_0$-presentable. Indeed consider the set $I := \{A_0 \subseteq A \mid |A_0| < \lambda_0\}$. For $A_0 \in I$, let $M_{A_0} := \cl^M (A_0)$. Note that $M_{A_0} \lta M$ by definition of $A$. Now $\seq{M_{A_0} : A_0 \in I}$ is a $\lambda_0$-directed system witnessing that $M$ is not $\lambda_0$-presentable (see Lemma \ref{mu-aec-present}).

  Finally, note that if $|U M| > \LS (\K)$ then by Fact \ref{internal-around-ls} also $r_{\K} (M) > \LS (\K) \ge \mu$. It follows directly that $r_{\K} (M) = |M|_{\cl}^+$.
\end{proof}

In view of Corollary \ref{inter-succ} and Theorem \ref{internal-cl}, is every $\mu$-AEC which admits intersections and has arbitrarily large models LS-accessible? In general, we do not know but we can show that this holds assuming GCH (compare with Corollary \ref{internal-size-cor-2}): 

\begin{thm}\label{intersec-ls-acc}
  Let $\K$ be a $\mu$-AEC which admits intersections. Let $\lambda > \LS (\K)$. Assume GCH (or at least $\SCH_{\mu, \lambda}$, see Definition \ref{sch-def}). If $\K$ has a model of cardinality at least $\lambda^+$, then $\K$ has a model of internal size $\lambda$.
\end{thm}
\begin{proof}
  If $\lambda$ is a successor, then Corollary \ref{inter-succ} gives the result so assume that $\lambda$ is a limit. Since $\lambda$ is limit and $\SCH_{\mu, \lambda}$ holds, we must have by Lemma \ref{mu-closed-sch} that $\lambda$ is $\mu$-closed.

  Let $N \in \K$ have cardinality at least $\lambda^+$ and let $A \subseteq U N$ have cardinality exactly $\lambda$. Let $M := \cl^{N} (A)$. We claim that $|M|_{\K} = \lambda$. By Theorem \ref{internal-cl}, $|M|_{\K} \le \lambda$. Assume for a contradiction that $|M|_{\K} < \lambda$. By Theorem \ref{size-ineq} and using that $\lambda$ is $\mu$-closed, $|U M| \le \left(|M|_{\K} + \LS (\K)\right)^{<\mu} < \lambda$. This is a contradiction since $|U M| \ge |A| = \lambda$.
\end{proof}
\begin{cor}
  Let $\K$ be a $\mu$-AEC which admits intersections and has arbitrarily large models. If $\SCH_{\mu, \ge \theta}$ holds for some $\theta$, then $\K$ is LS-accessible. In particular, if there is an $\omega_1$-strongly compact cardinal then any $\mu$-AEC admitting intersections that has arbitrarily large models is LS-accessible.
\end{cor}
\begin{proof}
  Immediate from Theorem \ref{intersec-ls-acc} and Remark \ref{sch-rmk}.
\end{proof}

\section{Examples}\label{examples-sec}

The common thread linking the results so far is the analysis, often under GCH or a suitable instance of SCH, of the relationship between internal size and cardinality in $\mu$-AECs, and the ways in which certain properties---the existence of models in each size, say---change when we toggle between these notions.  In this section, we study a series of examples that nicely capture this phenomenon, focusing in particular, on the ways in which the categoricity spectrum changes when we change our notion of size. 

Before we begin, though, two small cautionary examples illustrating that the behavior of small internal sizes can be quite wild:

\begin{example}\label{metric-example}
  In the category of complete metric spaces with isometries, the one point metric space $\{0\}$ is $\aleph_1$-presentable but \emph{not} $\aleph_0$-presentable (compare with Lemma \ref{pres-1}). Indeed, the set of reals $A = {0} \cup \{\frac{1}{n} \mid 1 \le n < \omega\}$ is the directed colimit of the spaces $A_m = \{\frac{1}{n} \mid 1 \le n \le  m\}$ for $m < \omega$ but the inclusion of $\{0\}$ into $A$ does not factor through any of the $A_m$'s. A similar argument shows that the empty space is the only $\aleph_0$-presentable complete metric space.
\end{example}

\begin{example}\label{wo-example}
  Let $\lambda$ be an infinite cardinal and let $\K$ be the ($\aleph_0$-)AEC of all well-orderings of order type at most $\lambda^+$, ordered by initial segment. Note that $\LS (\K) = \lambda$ and $\K$ admits intersections. Let $\alpha \le \lambda^+$. Then $|(\alpha, \in)|_{\K} = \cf{\alpha} + \aleph_0$. To see this, note that for $A \subseteq \alpha$, $\cl^{(\alpha, \in)} (A) = \alpha$ if and only if $A$ is cofinal in $(\alpha, \in)$, and use Theorem \ref{internal-cl}. In particular, $\K$ has no objects of singular internal size and when $\lambda$ is uncountable there are objects $M, N \in \K$ such that $M \lea N$ but $|M|_{\K} > |N|_{\K}$ (take $M := (\omega_1, \in)$ and $N := (\omega_1 + \omega, \in)$).
\end{example}

Looking at the categoricity spectrum, we now give a negative answer to Question \ref{categ-q-1}; that is, we give an example of the failure of eventual categoricity for $\mu$-AECs when categoricity is interpreted in terms of {\it cardinality}:

\begin{example}\label{hilbert-example}
  We consider a modification of an example of Makkai and Par\'e (\cite{makkai-pare} 3.4.2): rather than consider $\Hilb$, the category of (complex) Hilbert spaces and contractions, we consider the subcategory $\Hilb_0$, consisting of Hilbert spaces and their isometries.  We note, as an aside, that the isometries are precisely the regular monomorphisms in $\Hilb$.  The analysis of \cite{makkai-pare} holds even in this case---the only essential thing to note is that the norm on the colimit defined in their displayed equation (7) ensures that the colimit coprojections are themselves isometries.  We may conclude, then, that $\Hilb_0$ is $\aleph_1$-accessible and given that its morphisms are monomorphisms, it is equivalent to an $\aleph_1$-AEC (\cite{mu-aec-jpaa} 4.10). In fact, it is easy to check directly that it is an $\aleph_1$-AEC. Moreover, one can show (see \cite[3.4.4]{makkai-pare}) that in $\Hilb_0$, $|M|_{\Hilb_0}$ corresponds to the size of an orthonormal basis for $M$. Therefore $\Hilb_0$ is categorical in every internal size. 

However, by \cite[2.7]{BDHMP}, any infinite-dimensional Banach (and therefore Hilbert) space has cardinality $\lambda^{\aleph_0}$, for some infinite cardinal $\lambda$. Thus if $\gamma$ is such that $\lambda^{\aleph_0} = \aleph_{\gamma}$, $\alpha \le \gamma$ is least such that $\aleph_\alpha^{\aleph_0} = \aleph_\gamma$, and $\beta$ is such that $\gamma = \alpha + \beta$, then there are $(|\beta| + 1)$-many Hilbert spaces of cardinality $\aleph_{\gamma}$. In particular if $\alpha < \gamma$ then $\Hilb_0$ is not categorical in $\aleph_{\gamma}$ (but note that $\aleph_{\gamma}^{\aleph_0} = \aleph_{\gamma}$). Thus the $\aleph_1$-AEC of Hilbert spaces and isometries is indeed a counterexample to eventual categoricity in power, and a negative answer to Question \ref{categ-q-1}.

Assuming GCH, the situation is particularly clear: there are no models in cardinality $\lambda$ with $\cf{\lambda}=\aleph_0$, which is to be expected in an $\aleph_1$-AEC: we explicitly exclude such cardinals in the conjecture as formulated in Question \ref{categ-q-1}.  If $\lambda$ is not the successor of a cardinal of countable cofinality, there is a unique model of size $\lambda$---this is meaningful categoricity.  If, on the other hand, we consider $\lambda^+$ where $\cf{\lambda}=\aleph_0$, there are exactly two nonisomorphic models, one generated by a basis of size $\lambda$ and one by a basis of size $\lambda^+$.
\end{example}

\begin{example}
	In a general mAEC, one would expect to see similar failures of eventual categoricity in power: the requirement that the spaces underlying structures be complete introduces the same problem of cardinal exponentiation.  Eventual categoricity in terms of internal size---density character---is very much an open question, and there are few, if any, explicit computations of specific categoricity spectra.  We note, though, that eventual categoricity in density character has been shown to hold in the more restrictive classes of metric structures treated in \cite{shus837}.   These classes, which consist, roughly speaking, of complete approximately elementary (i.e. elementary-up-to-$\epsilon$)  submodels of a fixed compact, homogeneous metric structure $\sea$, are closely connected to a number of notions arising from continuous and homogeneous model theory. Specifically, they correspond to not-necessarily-Hausdorff \emph{metric compact abstract classes} (or \emph{metric cats}, \cite{bymcats}), and, by extension, are connected to the model theory of the logic of \emph{positive bounded formulas} (\cite{hensonio}) and \emph{continuous first order logic} (\cite{byuscontlog}, drawing on \cite{changkeis-cont}).
\end{example}

\subsection{Shelah's example}\label{shelah-example}

In the rest of this section, we study an example mentioned by Shelah in the introduction of \cite{sh1019}, which exhibits striking behavior whether we count isomorphism classes of models by cardinality or by internal size. The idea is to code the class of sufficiently-closed models of constructible set theory to obtain a $\mu$-AEC $\K^{\mu}$ that is categorical exactly in the cardinals of cofinality below $\mu$. We show, again, that the picture becomes quite different once one looks at internal sizes. In particular, while the categoricity spectrum in cardinalities alternates, $\K^{\mu}$ will not be categorical in \emph{any} internal size and in fact will have many models in each internal size.

We work in the language of set theory (i.e.\ it has equality and a binary relation $\in$). We will use the following definition of the constructible universe: let $L_0 = \emptyset$, $L_{\alpha + 1} = \Def (L_\alpha)$ (where $\Def (X)$ is the set of $Y \subseteq X$ such that there is a formula $\phi$ in the language of set theory and parameters $\ba \in \fct{<\omega}{X}$ such that $Y = \{x \in X \mid (X, \in) \models \phi[x, \ba]$), and $L_\beta = \bigcup_{\alpha < \beta} L_\alpha$ for $\beta$ limit. Finally let $L := \bigcup_{\alpha \in \OR} L_\alpha$. We work in Kripke-Platek (KP) set theory (see e.g.\ \cite[I.11]{devlinbook}: note that it includes the axiom of infinity). We will use the following facts about the theory of constructibility\footnote{The reader should be aware that there are some mistakes in Devlin's book \cite{devlinbook-review-stanley}. However, the results that we use in this paper are correct.}:

\begin{fact}[II.1.1.(vii) in \cite{devlinbook}]\label{l-fact-0}
  For any infinite ordinal $\alpha$, $|L_\alpha| = |\alpha|$.
\end{fact}
\begin{fact}[II.2.9 in \cite{devlinbook}]\label{l-fact-1}
  If $M$ is a well-founded model of KP + V = L, then there exists a unique ordinal $\alpha$ and a unique isomorphism $\pi: M \cong (L_\alpha, \in)$ 
\end{fact}
\begin{fact}[I.11.2, II.7.1 in \cite{devlinbook}]\label{l-fact-2}
  For any infinite cardinal $\lambda$, $(L_\lambda, \in)$ is a model of KP (and of $V = L$).
\end{fact}
\begin{fact}[II.5.5 in \cite{devlinbook}]\label{l-fact-3}
  For any ordinal $\alpha$, $\mathcal{P} (L_\alpha) \cap L \subseteq L_{|\alpha|^+ + \aleph_0}$. In particular, GCH holds in $L$.
\end{fact}

\begin{defin}
Let $K$ be the class of well-founded models of KP + V = L. Let $\K := (K, \preceq)$ (where $\preceq$ denotes the usual first-order elementary substructure). 
\end{defin}

Note that $K$ is axiomatizable by an $\Ll_{\aleph_1, \aleph_1}$-sentence (we only use an infinite quantifier to say that the universe is well-founded). Moreover, denoting by $I (\K, \lambda)$ the number of nonisomorphic models in $\K_\lambda$ we have:

\begin{lem}\label{aleph1-aec-lem}
  $\K$ is an $\aleph_1$-AEC. Moreover, it has no finite models and $I (\K, \lambda) = \lambda^+$ for every infinite cardinal $\lambda$.
\end{lem}
\begin{proof}[Proof of Lemma \ref{aleph1-aec-lem}]
  We first check that $\K$ satisfies the axioms from the definition of an $\aleph_1$-AEC. The only non-trivial ones are:
  \begin{itemize}
  \item Löwenheim-Skolem-Tarski axiom: by the $\Ll_{\omega, \omega}$ Löwenheim-Skolem-Tarski axiom, using that a subset of a well-founded model is still well-founded.
  \item Tarski-Vaught axioms: it is enough to check that for any $\aleph_1$-directed system $I$ and any $\seq{M_i : i \in I}$ increasing continuous in $\K$, $M := \bigcup_{i \in I} M_i$ is in $\K$. Note that $(M, \in^M)$ is well-founded. If not, there is a countable set $X \subseteq |M|$ witnessing it, and this countable set must be contained in $M_i$ for some $i \in I$, hence $M_i$ is ill-founded which is impossible by definition of $\K$. Now $M$ is a model of KP and $V = L$ by elementarity and the result follows.
  \end{itemize}

  For the moreover part, note first that the definition of KP that we use includes the axiom of infinity, so $\K$ has no finite models. Further, for any infinite cardinal $\lambda$, $(L_{\lambda^+}, \in)$ is in $\K$ by Fact \ref{l-fact-2}. By an easy argument using the $\Ll_{\omega, \omega}$ Löwenheim-Skolem-Tarski theorem, there is a club $C$ of ordinals $\alpha < \lambda^+$ such that $(L_{\alpha}, \in) \preceq (L_{\lambda^+}, \in)$. Since (for example by Fact \ref{l-fact-1}) for $\alpha \neq \beta$ both in $C$, $(L_\alpha, \in) \not \cong (L_\beta, \in)$, we obtain that $I (\K, \lambda) = |C| = \lambda^+$.
\end{proof}

We note that in $\K$, internal size and cardinality coincide. Unfortunately, this will fail when we pass to $\K^{\mu}$, which we will define presently.

\begin{thm}\label{k-internal-card}
  For any $M \in \K$, $|U M| = |M|_{\K}$.
\end{thm}
\begin{proof}
  By Lemma \ref{zfc-size}, $|M|_{\K} \le |U M|$. It remains to see that $|U M| \le |M|_{\K}$. If $\lambda = \aleph_0$, note that the definition implies that $|M|_{\K}$ is always infinite, so $|M|_{\K} \ge \lambda$. Now let $\lambda$ be a regular uncountable cardinal and assume that $|U M| \ge \lambda$. We show that $M$ is \emph{not} $\lambda$-presentable.

  We note that (by the $L_{\omega, \omega}$ Löwenheim-Skolem-Tarski theorem) $M$ can be obtained as a $\lambda$-directed union of elementary substructures of cardinality strictly less than $\lambda$. Thus there is a $\lambda$-directed system $\seq{M_i : i \in I}$ whose union is $M$ but for which there is no $i \in I$ with $M \lea M_i$. Therefore $M$ is not $\lambda$-presentable. 
\end{proof}

%Since $|L_\alpha| = |\alpha|$ for any infinite ordinal $\alpha$, it seems that presentability rank coincides with cardinality in $\K$. 
\begin{defin}\label{kstardef}
  Let $\mu$ be a regular cardinal. Let $K^\mu$ be the class of $M \in K$ isomorphic to $(L_\alpha, \in)$ that is such that for all $\beta < \alpha$ $[L_\beta]^{< \mu} \cap L \subseteq L_\alpha$. Let $\K^\mu := (K^\mu, \preceq)$.
\end{defin}

Notice that $\K = \K^{\aleph_0}$ and $\K^\mu$ is an $(\aleph_1 + \mu)$-AEC. Moreover, when $\mu > \aleph_0$ the behavior of its number of models is different from $\K$. We will use the following consequences of Jensen's convering Lemma (see \cite[Chapter V]{devlinbook}):

\begin{fact}\label{zsharp-fact}
  If $0^\sharp$ exists, then every uncountable cardinal is inaccessible in $L$. If $0^\sharp$ does not exist, then:
  \begin{enumerate}
  \item For any cardinal $\lambda$, $\cf{\lambda}^L \le \cf{\lambda} + \aleph_1$.
  \item For any singular cardinal $\lambda$, $\lambda^+ = (\lambda^+)^L$.
  \end{enumerate}
\end{fact}

\begin{thm}\label{card-nm}
  Let $\mu$ be a regular cardinal and let $\lambda \ge \mu$.

  \begin{enumerate}
  \item If $\cf{\lambda}^L < \mu$, then $I (\K^\mu, \lambda) = 1$.
  \item If $\cf{\lambda}^L \ge \mu$, then $I (\K^\mu, \lambda) \ge \lambda$.
  \item\label{cond-3} If $(\neg \exists \lambda_0 < \lambda^+ . \lambda^+ = \lambda_0^+ \land \cf{\lambda_0} < \mu)^L$, then $I (\K^\mu, \lambda) = \lambda^+$.
  \end{enumerate}
\end{thm}
\begin{proof}
  We build an ordinal $\delta \le \lambda^+$ and an increasing continuous chain $\seq{\alpha_i : i \le \delta}$, as follows. Set $\alpha_0 := \lambda$. If $i$ is limit, let $\alpha_i := \sup_{j < i} \alpha_j$. Now given $i < \lambda^+$, let $\beta > \alpha_i$ be least such that $(L_\beta, \in) \in \K$ and $[L_{\alpha_i}]^{<\mu} \cap L \subseteq L_\beta$. If $\beta \ge \lambda^+$, we set $\delta := i$ and stop. If $\beta < \lambda^+$, let $\alpha_{i + 1} := \beta$ and continue the induction.

  Now first assume that $\cf{\lambda}^L < \mu$. Then by Fact \ref{zsharp-fact} $0^\sharp$ does not exist and $(\lambda^+)^L = \lambda^+$. Since in $L$, $\lambda^{<\mu} = \lambda^+$, we must have that $\delta = 0$. Moreover it is easy to check that $(L_\lambda, \in) \in \K^\mu$. This proves that $I (\K^\mu, \lambda) = 1$.

  Assume now that $\cf{\lambda}^L \ge \mu$. Then $\delta \ge (\lambda^+)^L$. Letting $C$ be the set of limit points of the sequence $\seq{\alpha_i : i < \delta}$, we get that $(L_\alpha, \in) \in \K^\mu$ for any $\alpha \in C$, and $|C| \ge |(\lambda^+)^L| \ge \lambda$. Further if in $L$ $\lambda^+$ is not the successor of a cardinal of cofinality $\mu$, then $\delta = \lambda^+$ so $|C| = \lambda^+$.
\end{proof}

We can conclude that the categoricity spectrum of $\K^\mu$ alternates if $V = L$.  Indeed, the same is true if we merely assume that $0^\sharp$ does not exist and $\mu \ge \aleph_2$; for $\mu<\aleph_2$, the implications of the nonexistence of $0^\sharp$ are less clear. In contrast, if $0^\sharp$ exists, then $\K^\mu$ has many models in every cardinality.

\begin{cor}
  Let $\mu$ be a regular cardinal and let $\lambda \ge \mu$.
  \begin{enumerate}
  \item If $V = L$, then:
    \begin{align*}I(\K^\mu, \lambda) = 
  \begin{cases}
  	1 & \text{ if }\,\cf{\lambda} < \mu\\
  	\lambda^+ & \text{ if }\,\cf{\lambda} \ge \mu
  \end{cases}	
  \end{align*}
  \item If $0^\sharp$ exists, then $I (\K^\mu, \lambda) = \lambda^+$.
  \item If $0^\sharp$ does not exist, then:
    \begin{enumerate}
    \item If $\mu \ge \aleph_2$ and $\cf{\lambda} < \mu$, then $I (\K^\mu, \lambda) = 1$.
    \item If $\cf{\lambda} \in [\mu, \lambda)$, then $I (\K^\mu, \lambda) = \lambda^+$.
    \item If $\lambda$ is regular, then $I (\K^\mu, \lambda) \ge \lambda$.
    \end{enumerate}
  \end{enumerate}
\end{cor}
\begin{proof}
  By Fact \ref{zsharp-fact} and Theorem \ref{card-nm}.
\end{proof}

We now show that cardinality and internal sizes no longer coincide in $\K^\mu$ (compare with Theorem \ref{k-internal-card}). We will use that both $\K$ and $\K^\mu$ admit intersections:

\begin{fact}\label{l-admit-intersec}
  $\K$ and $\K^\mu$ admit intersections.
\end{fact}
\begin{proof}[Proof sketch]
  We show that $\K$ admits intersections. That $\K^\mu$ admits intersections will then directly follow from its definition.

  By \cite[II.3]{devlinbook}, there is a formula $\phi_{\text{WO}} (x, y)$ (in the language of set theory) such that for any $M \in \K$, $\phi_{\text{WO}} (x, y)$ well-orders the universe of $M$. Using $\phi_{\text{WO}}$, one can show that $M$ has \emph{definable Skolem functions}: for any formula $\phi (x, \by)$ there is a formula $\psi_{\phi} (x, \by)$ such that for any $M \in \K$ and $\bb \in \fct{<\omega}{U M}$:

  $$
  M \models \exists x \phi (x, \bb) \rightarrow \left(\exists! x (\psi_{\phi} (x, \bb) \land \phi (x, \bb))\right)
  $$

  The formula $\psi_{\phi} (x, \by)$ naturally induces a partial function $h_{\phi}: \fct{\ell (\by)}{U M} \rightarrow U M$ mapping $\by$ to $x$ whenever it exists. It is easy to see that for $M_0 \subseteq M$, $M_0$ is closed under the $h_\phi$'s if and only if $M_0 \preceq M$. Similarly, for any $A \subseteq U M$, $\cl^M (A)$ is nothing but the closure of $A$ under the $h_{\phi}$'s. Therefore $\cl^M (A) \preceq M$, as desired. In fact, we have shown that $\K$ is isomorphic (as a concrete category) to a \emph{universal} $\aleph_1$-AEC: the $\aleph_1$-AEC obtained by adding the $h_\phi$'s to each structure in $\K$.
\end{proof}

%% The following particular case of a well known set-theoretic argument will also come in handy:

%% \begin{fact}\label{embedding-lem}
%%   Assume that $V = L$. Let $\lambda$ be an infinite cardinal and let $M \in \K$. If $f:(L_{\lambda}, \in) \rightarrow M$ is an elementary embedding, then $f$ is the identity.
%% \end{fact}
%% \begin{proof}
%%   Since $V = L$, $L_{\lambda} = V_\lambda$. The result now follows from Claim 2 of \cite[3.4.5]{makkai-pare} and the fact that there are no measurable cardinals in $L$ \cite[17.1]{jechbook}.
%% \end{proof}

\begin{thm}\label{pres-example}
  Let $\mu$ be a regular cardinal and let $\lambda \ge \mu$. If $\lambda$ is regular or $\mu$-closed, then there are at least $\lambda^+$ non-isomorphic models in $\K^\mu$ of internal size $\lambda$.
\end{thm}
\begin{proof}
  Let $\theta := \lambda^{<\mu}$ and let $N := (L_{\theta^+}, \in)$. Note that $N \in \K$ by Fact \ref{l-fact-2} and $N \in \K^\mu$ by Fact \ref{l-fact-3}. We build an increasing chain of ordinals $\seq{\alpha_i : i < \lambda^+}$ such that:

  \begin{enumerate}
  \item $\alpha_i < \theta^+$ for all $i < \lambda^+$.
  \item $\alpha_i \notin \cl^N (\lambda \cup \{\alpha_j : j < i\})$.
  \end{enumerate}

  For $i < \lambda^+$, write $M_i := \cl^N (\lambda \cup \{\alpha_j : j < i\})$.

  \underline{This is possible}: Fix $i < \lambda^+$ and assume inductively that $\seq{\alpha_j : j < i}$ have been constructed. By Theorems \ref{size-ineq} and \ref{internal-cl}, $|U M_i| \le \theta$. Since $|U N| > \theta$, we can choose $\alpha_i$ as desired.
  
  \underline{This is enough}: We have that for $i < j < \lambda^+$, $M_i \not \cong M_j$ because $M_j$ has ``more ordinals'' than $M_j$ (see e.g.\ Fact \ref{l-fact-1}). If $\lambda$ is regular, let $C$ be the set of $i < \lambda^+$ such that $\cf{i} = \lambda$, otherwise let $C := \lambda^+$. We show that for $i \in C$, $|M_i|_{\K^\mu} = \lambda$. By definition, $|M_i|_{\ccl} \le \lambda$ (see Definition \ref{norm-cl}). Thus by Theorem \ref{internal-cl}, $r_{\K^\mu} (M_i) + \mu \le \lambda^+$, hence $|M_i|_{\K^\mu} \le \lambda$. If $\lambda$ is regular, then $\cf{i} = \lambda$ so $M_i = \bigcup_{j < i} M_j$ and $M_j \lta M_i$, so $M_i$ is not $\lambda$-presentable, i.e.\ $|M_i|_{\K^\mu} \ge \lambda$. Assume now that $\lambda$ is $\mu$-closed. If $|M_i|_{\K^\mu} < \lambda$, then by the proof of Theorem \ref{size-ineq}, $|U M_i| \le \left(|M_i|_{\K^\mu} + \mu)\right)^{<\mu} < \lambda$. Because $\lambda \subseteq U M_i$, we must have that $|U M_i| \ge \lambda$, a contradiction.
\end{proof}

\begin{question}
  In Theorem \ref{pres-example}, can we conclude that $\K^\mu$ has \emph{exactly} $\lambda^+$ non-isomorphic model of internal size $\lambda$? More generally, do we have that in any $\mu$-AEC $\K$ and any $\lambda > \LS (\K)$ there can be at most $2^\lambda$ non-isomorphic models of internal size $\lambda$?
\end{question}

Recalling Theorem \ref{card-nm}, we see that the internal sizes behave very differently from cardinalities: 

\begin{cor}\label{internal-nm}
  Let $\mu$ be a regular cardinal and let $\lambda \ge \mu$. Assume $\SCH_{\mu, \lambda}$. Then $\K^\mu$ has at least $\lambda^+$ non-isomorphic models of internal size $\lambda$.
\end{cor}
\begin{proof}
  If $\lambda$ is a successor, this is Theorem \ref{pres-example}. If $\lambda$ is a limit, then by $\SCH_{\mu, \lambda}$, $\lambda$ is $\mu$-closed so Theorem \ref{pres-example} also applies.
\end{proof}

\begin{question}
  Can we prove in ZFC that $\K^\mu$ is LS-accessible?
\end{question}

\appendix

\section{Locally multipresentable and polypresentable categories}\label{categ-sec}

We include a category-theoretic aside, concerning approximations of LS-accessibility in large locally multipresentable and locally polypresentable categories.  While these categories are not original to the authors (see \cite{diers} and \cite{lamarche-thesis}, respectively), the basic definitions and essential model-theoretic motivation can be found in \cite{multipres-pams}.  In short, a locally multipresentable category $\ck$ is an accessible category with all connected limits and, provided all of its morphisms are monomorphisms, it is, up to equivalence of categories, a universal $\mu$-AEC (\cite[5.9]{multipres-pams}).  Put another way, a locally multipresentable category is an accessible category with all multicolimits: rather than having a initial---colimit---cocone over each diagram $D$ in $\ck$, there is a family of multiinitial cocones, with the property that any compatible cocone admits a unique map from a unique member of the family. Locally polypresentable categories are accessible categories with wide pullbacks and, provided all of their morphisms are monomorphisms, are, up to equivalence, $\mu$-AECs admitting intersections (\cite[5.7]{multipres-pams}).  They can also be characterized as accessible categories with polycolimits, where the induced maps from the family of cocones described above are unique only up to isomorphism.

We are interested in the following version of the Löwenheim-Skolem theorem. The notion of LS-accessibility is introduced in \cite[2.4]{beke-rosicky} but weak LS-accessibility is new.

\begin{defin}\label{ls-acc-def}
  A category $\ck$ is called \emph{LS-accessible} if there exists a cardinal $\lambda$ such that for all $\lambda' \ge \lambda$, $\ck$ contains an object of internal size exactly $\lambda'$. We call $\ck$ \emph{weakly LS-accessible} if this holds only for all \emph{regular} $\lambda' \ge \lambda$.
\end{defin}

Beke and the second author \cite[4.6]{beke-rosicky} have shown that every large locally presentable category is LS-accessible. A modification of the argument shows:

\begin{thm}\label{lsacc}
Each large locally multipresentable category is weakly LS-accessible.
\end{thm}
\begin{proof}
Let $\ck$ be a large locally $\lambda$-multipresentable category. There is a $\lambda$-presentable object $A$ such that
the functor $\ck(A,-):\ck\to\Set$ takes arbitrarily large values because, otherwise, $\ck$ would be small. The functor
$U=\ck(A,-)$ preserves connected limits and $\lambda$-directed colimits. For every set $X$, the category $X\downarrow U$ 
is $\lambda$-accessible (see \cite{adamek-rosicky} 2.43) and has connected limits. Therefore, it has a multiinitial set of objects
$f_{Xi}:X\to UK_{Xi}$, $i\in I_X$. At first, we show that $|K_{Xi}|_\ck\leq|X|$ for each $i\in I_X$ and $\lambda\leq|X|$. 
Consider an $|X|^+$-directed colimit $l_j:L_j\to L$, $j\in J$ and a morphism $h:K_{Xi}\to L$. Since $U$ preserves $|X|^+$-directed
colimits, there is $j\in J$ such that $U(h)f_{Xi}=U(l_j)g$ for some $g:X\to UL_j$. Thus $g=U(\bar{g})f_{Xi'}$ for some $i'\in I_X$
and $\bar{g}:K_{Xi'}\to L_j$. Since $l_j\bar{g}:K_{Xi'}\to L$, we have $i'=i$. Then $l_j\bar{g}=h$ and, since this factorization is essentially unique, $K_{Xi}$ is $|X|^+$-presentable. Hence $|K_{Xi}|\leq|X|$.

Let $\lambda, |\ck_\lambda|<|X|$, where $\ck_\lambda$ is the set of morphisms between $\lambda$-presentable objects. We will show
that $|K_{Xi}|_\ck=|X|$ for some $i\in I_X$, which proves that $\ck$ is LS-accessible.
Assume that $|K_{Xi}|_\ck<|X|$ for $i\in I_X$. Let $\mu=\max\{\lambda,|K_{Xi}|_\ck\}^+$. Since $\mu\leq|X|$, $X$ is a $\mu$-directed
colimit of subsets $X_k$ of $X$ of cardinality $|X_k |<\mu$; $u_k:X_k\to X$ are the inclusions. For any $k$, there is a unique
$i_k\in I_{X_k}$ with a morphism $U(\bar{u}_k):K_{X_k,i_k}\to K_{Xi}$ such that $U(\bar{u}_k)f_{X_ki_k}=f_{Xi}u_k$. Analogously,
for each inclusion $u_{kk'}:X_k\to X_{k'}$ there is a morphism $\bar{u}_{kk'}:K_{X_ki_k}\to K_{X_{k'}i_{k'}}$ such that
$U(\bar{u}_{kk'})f_{X_ki_k}=f_{X_{k'}i_{k'}}u_{kk'}$. Morphisms $\bar{u}_{kk'}$ form a $\mu$-directed diagram and 
$\bar{u}_k:K_{X_ki_k}\to K_{Xi}$ a cocone from it. Let $t:\colim K_{X_ki_k}\to K_{Xi}$ be the induced morphism.  
Let $f:X\to U\colim K_{X_ki_k}$ be a unique mapping such that $fu_k=U(d_k)f_{X_ki_k}$ where $d_k$ are components of the colimit cocone. We have $U(t)f=f_{Xi}$ because 
$$
U(t)fu_k=U(t)U(d_k)f_{X_ki_k}=U(\bar{u}_k)f_{X_ki_k}=f_{Xi}u_k
$$
for each $k$. There is a unique $j\in I_X$ and a unique morphism $q:K_{Xj}\to\colim K_{X_ki_k}$ such that $U(q)f_{Xj}=f$. Since
$U(tq)f_{Xj}=U(t)f=f_{Xi}$, we have $j=i$ and $tq=\id_{K_{Xi}}$. Since
$$
U(q)U(\bar{u}_k)f_{X_ki_k}=U(q)f_{Xi}u_k=fu_k=U(d_k)f_{X_ki_k}
$$
we have $q\bar{u}_k=d_k$ and thus $qtd_k=q\bar{u}_k=d_k$. Hence $qt=\id_{\colim K_{X_ki_k}}$ and $t$ is an isomorphism with the inverse morphism $q$. Hence $\bar{u}_k:K_{X_ki_k}\to K_{Xi}$ is a colimit cocone. Thus $U(\bar{u}_k):UK_{X_ki_k}\to UK_{Xi}$ 
is a colimit cocone. Since $K_{ Xi}$ is $\mu$-presentable, there is $k$ and a morphism $r:K_{Xi}\to K_{X_ki_k}$ such that 
$\bar{u}_kr=\id_{K_{Xi}}$. 

Let $\delta=|I_X|$ and $\kappa=\beth_\lambda(\delta)$. Let $|Z|=\kappa$. Since 
$\cf{\kappa}=\lambda$, we have $|Z^X|=\kappa^\lambda>\kappa$. Choose $j\in I_Z$. For each $p:X\to Z$, there is $i_p\in I_X$ 
and $\bar{p}:K_{Xi_p}\to K_{Zj}$ such that $U(\bar{p})f_{Xi_p}=f_{Zj}p$. Since  $\delta<\kappa$, there is a subset $\cp\subseteq Z^X$ of cardinality $|\cp|>\kappa$ such that $i_p=i_q$ for each $p,q\in\cp$. Denote this common value of $i_p$ by $i$.
As shown above, there is a subset $X_k$ of $X$ of cardinality $\nu<\lambda$ such that $\bar{u}_k:K_{X_ki_k}\to K_{Xi}$ is a split
epimorphism. Since $\kappa^\nu=\kappa$, there is a subset $\cq\subset\cp$ such that $pu_k=qu_k$ for each $p,q\in\cq$. For $p,q\in\cq$,
we have
$$
U(\bar{p})U(\bar{u}_k)f_{X_ki_k}= U(\bar{p})f_{Xi}u_k=f_{Zj}pu_k=f_{Zj}qu_k=U(\bar{q})f_{Xi}u_k=U(\bar{q})U(\bar{u}_k)f_{X_ki_k}
$$
and thus $\bar{p}\bar{u}_k=\bar{q}\bar{u}_k$. Hence $\bar{p}=\bar{q}$. Since 
$$
f_{Zj}p=U(\bar{p})f_{Xi}=U(\bar{q})f_{Xi}=f_{Zj}q
$$
$f_{Zj}$ is not a monomorphism.
 
Since the sets $UM$ are arbitrarily large, there is a monomorphism $v:Z\to UM$. Since $v$ factorizes through $Uf_{Zj}$ for some
$j\in I_Z$, this $f_{XZj}$ is a monomorphism. Thus $|K_{Xi}|_\ck=|X|$.
\end{proof}

Assuming that all morphisms are monomorphisms, we can generalize the argument further to locally polypresentable categories:

\begin{thm}\label{lsacc1}
Each large locally polypresentable category with all morphisms mo\-no\-mor\-phisms is weakly LS-accessible.
\end{thm}
\begin{proof}
We will follow the proof of \ref{lsacc}. The first paragraph is unchanged, only the essential unicity at the end follows
from the fact that morphisms in $\ck$ are monomorphisms. In the second paragraph, we do not know that 
$\bar{u}_{k'}\bar{u}_{kk'}=\bar{u}_k$. But we know that there is an isomorphism $h_{kk'}:K_{X_ki_k}\to K_{X_ki_k}$ such that
$\bar{u}_{k'}\bar{u}_{kk'}h_{kk'}=\bar{u}_k$. Since $\ck$-morphisms are monomorphisms, $\tilde{u}_{kk'}=\bar{u}_{kk'}h_{kk'}$ 
form a $\mu$-directed diagram and $\bar{u}_k:K_{X_ki_k}\to K_{Xi}$ a cocone from it. Let $t:\colim K_{X_ki_k}\to K_{Xi}$ 
be the induced morphism and $f:X\to U\colim K_{X_ki_k}$ the induced mapping. We have $U(t)f=f_{Xi}$ because 
$$
U(t)fu_k=U(t)U(d_k)f_{X_ki_k}=U(\bar{u}_k)f_{X_ki_k}=f_{Xi}u_k
$$
for each $k$. Let $q:K_{Xi}\to\colim K_{X_ki_k}$ be a morphism such that $U(q)f_{Xi}=f$. We have $U(tq)f_{Xi}=f_{Xi}$. Thus 
there is an isomorphism $h:K_{Xi}\to K_{Xi}$ such that $tq=h\cdot\id_{K_{Xi}}=h$. Thus $t$ is a split epimorphism and,
since it is a monomorphism, it is an isomorphism. Hence $\bar{u}_k:K_{X_ki_k}\to K_{Xi}$ is a colimit cocone.
Since $K_{Xi}$ is $\mu$-presentable, there is $k$ and a morphism $r:K_{Xi}\to K_{X_ki_k}$ such that $\bar{u}_kr=\id_{K_{Xi}}$.
As a split epimorphism, $\bar{u}_k$ is an isomorphism.  

As in the proof of \ref{lsacc}, let $\alpha=|X|$ be regular. Let $\iota_i$ be the number of isomorphisms of $K_{X_ki_k}$, $i\in I_X$.
Now, $\delta$ will be the maximum of $|I_X|$ and $sup_{i\in I_X}\iota_i$. For $p,q\in\cq$, we do not get 
$\bar{p}\bar{u}_k=\bar{q}\bar{u}_k$, but we get an isomorphism $h_q:K_{X_ki_k}\to K_{X_ki_k}$ such that $\bar{p}=h_q\bar{q}$. 
Since $\cq>\iota_i$, there is a subset $\cq_o$ of $\cq$ such that $h_{q_1}=h_{q_2}$ for $q_,q_2\in\cq_0$. Hence $\bar{p}=\bar{q}$
for $p,q\in\cq_0$.  
\end{proof}

We note that this amounts to an alternative---purely category-theoretic---proof of Corollary~\ref{inter-succ}.

\bibliographystyle{amsalpha}
\bibliography{categ-accessible}

\end{document}